\documentclass[10pt]{amsart}

\pdfoutput=1
\usepackage{enumerate,amsthm,amsmath}
\usepackage{amssymb}
\usepackage{graphicx}
\usepackage{xspace}
\usepackage{tikz}
\usetikzlibrary{patterns}
\usetikzlibrary{arrows}
\usepackage{comment}

\newcommand{\etal}{et.~al.}

\newenvironment{descr}
{\begin{list}{$\bullet$}{\topsep0mm \labelwidth5mm \leftmargin7mm %
  \itemsep1pt plus 2pt \topsep3pt \parsep1pt plus 4pt \labelsep2mm}}
{\end{list}}

\usepackage[%
   pdfpagemode=UseNone,
   bookmarks=true,
   pagebackref=false, 
   colorlinks,
   linkcolor=blue,
   anchorcolor=blue,
   citecolor=blue,
   filecolor=blue,
   pagecolor=blue,
   urlcolor=blue
   ]{hyperref}

\usepackage{stmaryrd}

\numberwithin{equation}{section}

\newcommand{\note}[1]%
{\noindent\centerline{\fbox{\parbox{.9\textwidth}{\textbf{#1}}}}}
\newcommand{\snote}[1]%
{\fbox{\textbf{#1}}}

\usepackage{accents}
\newcommand{\Hos}{\ensuremath{\accentset\circ H^s}}

\newcommand{\step}[1]{\noindent\raisebox{1.5pt}[10pt][0pt]{\tiny\framebox{$#1$}}\xspace}

\renewcommand{\P}{\mathcal{P}}
\newcommand{\T}{\ensuremath{\mathcal T}}
\newcommand{\B}{\ensuremath{\mathcal B}}
\newcommand{\E}{\ensuremath{\mathcal E}}
\newcommand{\Tz}{\T_z}
\newcommand{\Ez}{\E_z}
\newcommand{\VT}{\ensuremath{{\mathbb V}_\T}}

\newcommand{\norma}[1]{\left\| {#1} \right\|}

\newcommand{\scp}[2]{\left\langle{#1}\, , \,{#2}\right\rangle}
\newcommand{\R}{\mathbb R}

\newcommand{\NN}{\mathbb N}
\DeclareMathOperator{\supp}{supp}

\DeclareMathOperator{\dist}{dist}
\DeclareMathOperator{\tr}{tr}
\renewcommand{\SS}{\mathbb S}

\newcommand{\ol}{\overline}

\newcommand{\e}{\ensuremath{u - U}}

\newcommand{\N}{\ensuremath{{\mathcal N}_\ell}}
\newcommand{\V}{\ensuremath{{\mathcal V}}}

\newcommand{\grad}{\nabla}

\newcommand{\vphi}{\ensuremath{\varphi}}
\newcommand{\D}{\displaystyle}

\newcommand{\That}{{\hat T}}

\newcommand{\xhat}{{\hat x}}
\newcommand{\yhat}{{\hat y}}
\newcommand{\psihat}{{\hat \psi}}
\newcommand{\phihat}{{\hat \phi}}
\newcommand{\tw}{\widetilde w}

\DeclareMathOperator{\diam}{diam}

\newtheorem{theorem}{Theorem}[section]

\newtheorem{lemma}[theorem]{Lemma}

\newtheorem{corollary}[theorem]{Corollary}

\newtheoremstyle{examplestyle}
  {6pt}
  {6pt}
  {}
  {}
  {\bfseries}
  {.}
  {1em}
  {}
\theoremstyle{examplestyle}
             \newtheorem{example}[theorem]{Example}
             \newtheorem{remark}[theorem]{Remark}
\newtheoremstyle{algorithmstyle}
  {6pt}
  {6pt}
  {\ttfamily}
  {}
  {\bfseries}
  {.}
  {1em}
  {}
\theoremstyle{algorithmstyle}

\begin{document}

\title[A~posteriori error estimates with point sources] %
  {A posteriori error estimates with point
  \\
  sources  in fractional Sobolev spaces}

\author{Fernando D.\ Gaspoz}
\address{Institut f\"ur Angewandte Mathematik und Numerische Simulation, Universit\"at Stuttgart. Pfaffenwaldring 57, D-70569 Stuttgart, Germany}
\email{fernando.gaspoz@ians.uni-stuttgart.de}

\author{Pedro Morin}
\address{Instituto de Matem\'atica Aplicada del Litoral (CONICET-UNL) and Facultad de Ingenier\'ia Qu\'imica, Universidad Nacional del Litoral, Santa Fe, Argentina}
\curraddr{IMAL. Colectora Ruta Nac.~168, Paraje El Pozo, 3000 Santa Fe, Argentina}
\email{pmorin@santafe-conicet.gov.ar}

\author{Andreas Veeser}
\address{Dipartimento di Matematica, Unversit\`a degli Studi di Milano, Via C.~Saldini 50, 20131~Milano, Italy}
\email{andreas.veeser@unimi.it}

\begin{abstract}
We consider Poisson's equation with a finite number of weighted Dirac masses as a source term, together with its discretization by means of conforming finite elements.  For the error in fractional Sobolev spaces, we propose residual-type a~posteriori estimators with a specifically tailored oscillation and show that, on two-dimensional polygonal domains, they are reliable and locally efficient.  In numerical tests, their use in an adaptive algorithm leads to optimal error decay rates. 
\end{abstract}



\maketitle

\section{Introduction}
%
%
%

We consider the problem
\begin{equation}\label{Problem}
 -\Delta u = \sum_{j=1}^N \alpha_j \delta_{x_j} \quad \text{in } \Omega,
\qquad
  u = 0 \quad \text{on } \partial\Omega,
\end{equation}
where $\Omega\subset\R^2$ is a polygonal domain with Lipschitz boundary $\partial\Omega$ and, for any
$j=1,2,\dots , N$, each $\alpha_j \delta_{x_j}$ is a point source given by $\alpha_j \in \R$ and the Dirac measure $\delta_{x_j}$ at $x_j \in \Omega$.
Such point sources are a useful idealization in modeling and appear in various applications: for instance, in modeling the effluent discharge in aquatic media~\cite{ABR2}, in reaction-diffusion problems taking place in domains of different dimension~\cite{DAngelo-Quarteroni}, and in modeling the electric field generated by point charges~\cite{ACRV}.

Since $\Omega\subset\R^2$, point sources induce singularities of the type $\log |\cdot-x_j|$. In particular, they do not belong to $H^{-1}(\Omega)$ and so the solution $u$ of \eqref{Problem} is not in $H^1(\Omega)$.  Hence, the variational formulation within $H^1(\Omega)$ cannot be used for \eqref{Problem} and the usual approach to a~posteriori error estimation in the energy norm is not directly applicable.  

Although $u\not\in H^1(\Omega)$, it can be approximated by finite element methods arising from the variational formulation within $H^1(\Omega)$: e.g., consider the finite element space $\VT$ of continuous functions that are piecewise polynomial up to degree $\ell$ over a triangulation $\T$ of $\Omega$. Then the Galerkin approximation $U\in\VT$ given by
\begin{equation}
\label{Galerkin}
 \int_\Omega \grad U \cdot \grad V
 =
 \sum_{j=1}^N \alpha_j V(x_j),
\quad
 \forall V \in \VT,
\end{equation}
is well-defined thanks to the continuity of the functions in $\VT$.

In view of $u\not\in H^1(\Omega)$, the error of the approximation $U$ has to measured in a norm weaker than the $H^1$-norm.  For quasi-uniform meshes, Babu\v{s}ka \cite{Babuska} and Scott \cite{Scott} derived a~priori estimates for the error in the $H^s$-norm, $s\in[0,1)$.  Exploiting that the singularity of point sources is known, Eriksson~\cite{Eriksson1} proves a~priori estimates for the $L^1$- and $W^{1,1}$-norms that show the advantage of suitably graded meshes.

 
Several a~posteriori error estimates, which may be used to direct an adaptive algorithm, are also available.  Araya \etal~\cite{ABR,ABR2} derived a~posteriori estimates for $L^p$-norm, $p\in(1,\infty)$, and the $W^{1,p}$-norm, $p\in(p_0,2)$, where $p_0\in[1,2)$ depends on $\Omega$.  More recently, Agnelli \etal~\cite{AGM} obtained a~posteriori estimates for the weighted Sobolev norms introduced by D'Angelo~\cite{DAngelo}.

\medskip In this article, we analyze a~posteriori estimators for the error between $u$ and its Galerkin approximation $U$ in the $H^{1-\theta}$-norm, $0<\theta<\frac{1}{2}$.  It is based upon the variational formulation within $H^{1-\theta}(\Omega)$ of Ne\v{c}as~\cite{Necas}.  The error indicators are given by
\begin{equation*}
 \eta_{T , \theta}
 =
 \left[
  h_T^{2+2\theta} \norma{\Delta U}_{L^2 (T )}^2
  +
  h_T^{1+2\theta}
   \norma{\llbracket \nabla U\rrbracket}_{L^2 (\partial T)}^2
  \right]^{\frac{1}{2}}, 
\qquad
  T \in\T,
\end{equation*}
where $h_T=|T|^{1/2}$ stands for the local meshsize and $\llbracket \nabla U\rrbracket$ denotes the jump of the normal derivative across interelement edges.  Our main result is the reliability and the local and global efficiency of these a~posteriori error estimators.  More precisely:

\medskip\noindent\textbf{Main Result.}
\emph{For any $0 < \theta < \frac{1}{2}$, the error of $U\in\VT$ is bounded from above and below in the following manner:
\begin{gather*}
 \norma{u-U}_{H^{1-\theta}(\Omega)}
 \le
 C_{U} \Big(
  \sum_{T \in \T } \eta_{T,\theta}^2
 \Big)^{\frac{1}{2}}\
 +
 \xi_{\theta},
\\
 \eta_{T,\theta}
 \le
 C_{L} \norma{u-U}_{H^{1-\theta}(\omega_T)}
 \;\; (T \in \T),
\qquad
	\Big(\sum_{T \in \T } \eta_{T,\theta}^2 \Big)^{\frac{1}{2}}
	\le C_{L} \norma{u-U}_{H^{1-\theta} ( \Omega)} .
\end{gather*}
The quantity $\xi_{\theta}$ is an oscillation-type term (see Section~\ref{S:upper-bound}) and the constants $C_{U}$, $C_{L}$ depend on $\theta$, $\Omega$, the polynomial degree $\ell$, and the minimum angle of the underlying triangulation.  Moreover, $\omega_T$ is the patch of all the elements sharing a
side with $T$.} 

\medskip\noindent The following comments are in order:
\begin{descr}
\item The indicators $\eta_{T,\theta}$, $T\in\T$, coincide with the standard residual ones for Laplace's equation and so do not depend on the point sources in \eqref{Problem}.
\item The term $\xi_{\theta}$, which depends on the point sources, vanishes under certain conditions, e.g., if the point sources for every star of the triangulation have the same sign; see Remark \ref{R:osc-ref}.  Noteworthy, if this is not already given on the initial grid, it can be always met after a finite number of refinements.  The term $\xi_{\theta}$ is thus a non-standard oscillation term.
\item In the many cases in which $\xi_{\theta}$ vanishes, the error of $U$ is encapsulated only with the help of the approximate solution, without invoking data.
\end{descr}

The rest of the article is organized as follows: Section~\ref{S:posedness} reviews fractional Sobolev spaces, the variational formulation of \eqref{Problem} within $H^{1-\theta}$, its dual problem, and its finite element discretization.  In Section~\ref{S:apost-analysis}, we give a complete definition of the oscillation term $\xi_\theta$ and prove the presented main result. Finally, in Section~\ref{S:experiments}, we test the indicators $\eta_{T,\theta}$, $T\in\T$, in an adaptive algorithm.

Throughout this article, $a \lesssim b$ will denote $a \le C b$ with some constant whose dependence will be stated whenever it is not clear from the context. We will use $a \simeq b$ to denote $a \lesssim b$ and $b \lesssim a$.

\section{Continuous, dual, and discrete problems}
\label{S:posedness}

In this section, we review fractional order Sobolev spaces, as well as the well-posedness of \eqref{Problem} and its dual problem in such spaces.  This will be instrumental for building up the proofs of the a posteriori error bounds.  Moreover, we recall the finite element discretization of \eqref{Problem} and associated notation.


\subsection{Fractional order Sobolev spaces}

Let $G$ be a bounded open set of $\R^d$, $d \in \NN$, with a Lipschitz boundary.  We use the following notation for the (semi)norms of the usual (Hilbertian) Sobolev spaces $H^n(\Omega)$ of integer order $n\in\NN_0$:
\begin{gather*}
 \left\| \phi \right\|_{0,G}
 :=
  \left| \phi\right|_{0,G}
 :=
 \left( \int_G |\phi|^2 \, dx \right)^{\frac12},
\qquad
 \left| \phi\right|_{n,G} 
 :=
 \bigg(
  \sum_{|\beta|= n} \big\|D^{\beta} \phi\big\|_{0,G}^2
 \bigg)^{\frac12},
\\
 \left\| \phi \right\|_{n,G}
 :=
 \bigg(
  \sum_{|\beta|\le n} \big\|D^{\beta} \phi\big\|_{0,G}^2
 \bigg)^{\frac12} 
 =
 \bigg( \sum_{k=0}^n \left| \phi\right|_{k,G}^2 \bigg)^{1/2},
\end{gather*}
where, for any multi-index $\beta = (\beta_1,\beta_2,\dots,\beta_d)\in \NN_0^d$,
its length is defined by $|\beta| := \beta_1 + \beta_2 + \dots + \beta_d$ and $D^\beta \phi$ denotes the weak $\beta$-derivative of $\phi$, with the convention $D^{(0,\dots,0)} \phi = \phi$. 
If $s > 0$ is not an integer, we write $s = n+t$ with $n\in\NN_0$ and $0 < t < 1$ and define the norm of $H^{s}(\Omega)$ by $\left\| \phi \right\|_{s,G}^2
:= 
\left\| \phi \right\|_{n,G}^2  +  \left | \phi \right |_{s,G}^2 
$, with 
\[
\left | \phi \right |_{s,G}
  := \bigg[
       \sum_{|\beta|=n}
  \int_G \int_G \frac{|D^\beta \phi(x)-D^\beta \phi(y)|^2}{|x-y|^{d+2t}}\, dx \, dy
     \bigg]^{\frac12}.
\]
The space $H_0^s(G)$ is the completion of $C_0^\infty(G)$ under the norm $\| \cdot \|_{s,G}$.

\smallskip The following lemma summarizes some basic properties of the fractional Sobolev spaces $H^s(\Omega)$, which can be found, e.g., in~\cite{Hack}.
\begin{lemma}[Fractional Sobolev spaces]
\label{L:basic-properties}
Let $G$ be a bounded open subset of $\R^d$, $d\in\NN$, with a Lipschitz boundary.  We have:
\begin{enumerate}[\rm(i)]
%
%
\item\label{it:shift}(Shift) If $|\beta| \le s$, then $\phi \in H^s(G)$ entails $D^\beta \phi \in H^{s-|\beta|}(G)$. 
\item\label{it:trace}(Trace) If $s > \frac12$, there exists a constant $C$ such that
\begin{equation}\label{E:trace}
 \| \phi \|_{0,\partial G}
 :=
 \| \phi \|_{L^2(\partial G)}
 \le
 C \| \phi \|_{s,G},
\quad \text{for all }\phi \in C^\infty(\overline{G}),
\end{equation}
and thus the \emph{trace operator}, defined in $C^\infty(\overline{G})$ as $\tr \phi = \phi_{|\partial G}$ can be extended to be a continuous operator from $H^{s\vphantom{\beta}}(G)$ into $L^2(\partial G)$.  In particular, we have $\tr \phi = 0$ for all $\phi \in H_0^s(G)$ and $s > \frac{1}{2}$.
%
%
\item\label{it:sobolev-embedding}(Sobolev embedding) If $s > \frac{d}2$ then $H^s(G)$ is continuously embedded into $C^{k,t}(\ol G)$, with $k = \lceil s-\frac{d}2 \rceil -1$ and $s - \frac{d}2 = k + t$. More precisely, if $s-\frac{d}2 = k + t$ with $k \in \NN_0$ and $0 < t \le 1$, then for any function $\phi \in H^s(G)$ there exists a function $\tilde \phi \in C^k(\ol G)$ such that $\phi = \tilde \phi$ a.e.\ in $G$ and 
\[
\max_{|\beta| \le k} 
 \| D^\beta \tilde\phi \|_{L^\infty(G)} + 
\max_{|\beta| = k}  \sup_{\substack{x,y\in G\\ x\neq y}} \frac{|D^\beta\tilde\phi(x) - D^\beta\tilde\phi(y)|}{|x-y|^t}
\le C \| \phi \|_{s,G}.
\]
\end{enumerate}
The constants $C$ appearing in \eqref{it:trace} and \eqref{it:sobolev-embedding} depend only on the set $G$, the dimension $d$, the smoothness order $s$, but are otherwise independent of $\phi$.
\end{lemma}

\smallskip We shall need several inequalities involving fractional Sobolev seminorms on domains (bounded, connected and open sets).  The first one is the counterpart of the classical Poincar\'e inequality for functions in
\begin{equation*}
 \Hos(G)
 :=
 \left\{ \phi \in H^s(G) : \int_G \phi = 0 \right\}.
\end{equation*}
We shall use it to derive further inequalities, as well as for the proof of the lower a~posteriori error bounds.

\begin{lemma}[Fractional Poincar\'e inequality]
\label{L:Poincare-meanvalue}
If $0<s<1$ and $G$ is a bounded domain of $\R^d$, there is a constant $C_P$ depending on $s$ and $G$ such that 
\[
 \| \phi \|_{0,G}^2
 \le
 C_P 
 |\phi|_{s,G}^2
,\qquad\text{for all } \phi\in \Hos(G).
\]
\end{lemma}

\begin{proof}
The simple proof of Faermann \cite[Lemma~3.4]{F} (which readily generalizes to $d$-dimensional domains) shows
\begin{equation}
\label{CP<=}
 C_P \le \frac{1}{2}\frac{\diam(G)^{d+2s}}{|G|},
\end{equation}
where $\diam(G)$ and $|G|$ stand for the diameter and the Lebesgue measure of $G$, respectively.  The dependence on $s$ can be identified more precisely, see, e.g., Bourgain \etal\ \cite{Bourgain.Brezis.Mironescu:02}, but \eqref{CP<=} suffices for our purposes.
\end{proof}

We need also the following generalization of the Friedrichs inequality.  It is related to the well-posedness of \eqref{Problem} and the definiteness of the error notion considered in the a~posteriori analysis below.

\begin{lemma}[Fractional Friedrichs inequality]
\label{L:Poincare}
Let $s > 1/2$ and $G\subset\R^d$ be a bounded domain with Lipschitz boundary $\partial G$.  We have 
\begin{equation}\label{Poincare}
 \| \phi \|_{0,G}
 \le
  C_{F} |\phi|_{s,G},
\qquad\text{for all } \phi\in H_0^s(G),
\end{equation}
where the constant $C_{F}$ depends on $d$, $s$ and $G$.
\end{lemma}

\begin{proof}
We distinguish different cases for $s$.

\step{1} If $s \in \NN$, the claimed inequality is the Friedrichs inequality for Sobolev spaces of integer order; see, e.g., \cite[Ch.~II, 1.7]{B}.
  
\step{2} Let $\frac12 < s <1$ and fix $\phi\in H_0^s(G)$.  For any constant $c\in\R$, we write 
\[
 \phi
 =
 \phi - \frac{1}{|\partial G|} \int_{\partial G} \phi
 =
 (\phi-c) - \frac{1}{|\partial G|} \int_{\partial G} (\phi-c),
\]
where $|\partial G|$ denotes the $(d-1)$-dimensional Hausdorff measure in $\R^d$.  Thus, the Cauchy-Schwarz inequality on $\partial G$ and the trace theorem \eqref{E:trace} imply
\[
 \| \phi \|_{0,G}
 \le
 \| \phi - c\|_{0,G}
 +
 \left( \frac{|G|}{|\partial G|} \right)^{1/2}
  \| \phi - c\|_{0,\partial G}
 \le
 \| \phi - c\|_{0,G}
 +
 C \| \phi - c\|_{s,G}
\]
with $C$ depending on $G$.  We choose $c=|G|^{-1}\int_G \phi$ and 
obtain the claimed inequality in this case with the help of Lemma \ref{L:Poincare-meanvalue}.

\step{3} Let $k < s < k+1$ with $k\in\NN$ and assume, without loss of generality, that $\phi \in C_0^\infty(G)$.  Observe that, for any multi-index $\beta$ with $|\beta| = k$, we have $\int_G D^\beta \phi = 0$ and, therefore, Lemma~\ref{L:Poincare-meanvalue} yields
\[
 \| D^\beta \phi \|_{0,G}
 \le
 C |D^\beta \phi|_{s-k,G},
\quad\text{whence}\quad
 | \phi |_{k,G}
 \le
 C | \phi |_{s,G}.
\]
The claimed inequality then follows from Step~\step{1}.
\end{proof}

For any $s > 0$, we define $H^{-s}(G)$ as the topological dual space of $H_0^s(G)$. In other words: $H^{-s}(G)$ is the set of linear functionals $\psi :H_0^s(G) \to \R$ satisfying
\[
| \langle \psi, \phi \rangle | \le C \| \phi \|_{s,G}, \quad\text{for all }\phi \in H_0^s(G),
\]
for some constant $C$ independent of $\phi$. Hereafter, the symbol $\scp{\cdot}{\cdot}$ indicates the result of the application of a functional to a function in its domain of definition.

The spaces $H^{-s}(G)$ with $s>\frac{1}{2}$ are of particular interest for us.  It is convenient to equip them with norms that have simple scaling properties, like seminorms.  Lemma \ref{L:Poincare} ensures  that, for such $s$, the seminorm $|\cdot|_{s,G}$ is a norm in $H_0^s(G)$, equivalent to $\| \cdot \|_{s,G}$.  We thus can define a suitable norm for $H^{-s}(G)$ with $s>\frac{1}{2}$ by duality:
\[
 \left\| \psi \right\|_{-s,G}
 :=
 \sup_{\phi\in H_0^s(G)}
  \frac{ \langle \psi , \phi \rangle}
   {\left| \phi \right |_{s,G}}
 =
 \sup_{\phi\in C_0^\infty(G)}
  \frac{ \langle \psi , \phi \rangle}
   {\left| \phi \right |_{s,G}},
\quad \text{ for }\psi \in H^{-s}(G).
\]
We then have also
\begin{equation}
\label{norm-duality}
 \D \left| \phi \right|_{s,G} 
 =
 \sup_{\psi \in H^{-s}(G)}
  \frac{\langle \psi, \phi \rangle}
     {\left\| \psi \right\|_{-s,G}},
\quad\text{ for }\phi \in H_0^s(G).
\end{equation}
For the sake of brevity, we will write $H^s$, $H_0^s$, $\| \cdot \|_s$, to denote $H^s(\Omega)$, $H_0^s(\Omega)$, $\| \cdot \|_{s,\Omega}$, respectively.

\smallskip Finally, we need a Bramble--Hilbert-type inequality; it will be useful in deriving the a~posteriori upper error bound.  For this purpose, it is sufficient to consider triangular domains.   Moreover, in view of the following lemma, it is sufficient to derive such, and other, inequalities for the reference triangle $\Hat{T}$ in $\R^2$ given by the vertices $(0,0)$, $(1,0)$, and $(0,1)$.

\begin{lemma}[Scaling properties of Sobolev norms]
\label{L:scaling}
Let $T \subset \Omega$ be a triangle, set $h_T:=|T|^{1/2}$ and let $F(\xhat) = A \xhat + b$ be an affine bijection such that $F(\That)=T$.
\begin{enumerate}[\rm(i)]
\item Assume that the functions $\phi:T\to\R$ and $\phihat:\That\to\R$ satisfy $\phihat=\phi\circ F$.  Then, for any $s\ge0$, we have $\phi\in H^s(T)$ if and only if $\phihat\in H^s(\That)$ and
\[
 | \phi |_{s,T}
 \simeq
 h_T^{1-s} |\phihat|_{s,\That}.
\]
\item Assume that the distributions $\psi:C^\infty_0(T)\to\R$ and $\psihat:C^\infty_0(\That)\to\R$ satisfy $\langle\psihat,\phi\circ F\rangle = |\det(A)|^{-1} \langle\psi,\phi\rangle$ for all $\phi \in C_0^\infty(T)$. Then, for any $s>\frac12$, we have $\psi\in H^{-s}(T)$ if and only if $\psihat\in H^{-s}(\That)$ and
\[
 \norma{\psi}_{-s,T}
 \simeq
 h_T^{1+s} \norma{\psihat}_{-s,\That}.
\]
\end{enumerate}
The hidden constants depend only on $s$ and the minimum angle of $T$.
\end{lemma}
\begin{proof}
Up to constants depending on the minimum angle of $T$, we have
\[
 |\det(A)| \simeq h_T^2,
\quad
 \|A\| \simeq h_T,
\quad\text{and}\quad
 \|A^{-1}\|\simeq h_T^{-1}.
\]
The case $s\in\NN_0$ is well known, see, for instance, \cite[Ch. II, 6.6]{B}.  Next, consider $0< s < 1$. Given $\phi \in H^s(T)$, we derive 	\begin{align*}
 \left| \phi \right|_{s,T}^2  
 &=
 \int_{T}\int_{T} \frac{|\phi(x)-\phi(y)|^2}{|x-y|^{2+2 s}} \,dx\,dy 
 =
 \int_{\That}\int_{\That}
  \frac{|\phihat(\xhat)-\phihat(\yhat)|^2}
    {|A (\xhat - \yhat )|^{2+2 s}} |\det(A)|^2 \,d\xhat \,d\yhat 
\\
 &\le
  \frac{|\det(A)|^2}{\| A^{-1} \|^{-(2+2s)}}
  \int_{\That}\int_{\That}
   \frac{|\phihat(\xhat)-\phihat(\yhat)|^2}
    {|\xhat - \yhat|^{2+2 s}}  \,d\xhat \,d\yhat 
 \lesssim
 h_T^{2-2s} |\phihat|_{s,\That}^2,
\end{align*}
where we have used that $|A(\xhat-\yhat)| \ge \| A^{-1} \|^{-1} |\xhat-\yhat|$.  The opposite inequality follows in a similar manner using $|A(\xhat - \yhat)| \le \|A\| |\xhat - \yhat|$ and (i) is thus verified also for $0<s<1$.
	
For the case $k < s < k+1$ with $k \in \NN$, we observe 
\begin{equation*}
 \sum_{|\beta|=k}
  |D^\beta \phi(x)- D^\beta \phi(y)|
 \simeq
 \|A^{-1}\|^k
 \sum_{|\beta|=k}
  |D^\beta \phihat(\hat{x}) - D^\beta \phihat(\hat{x})|
\end{equation*}
and then proceed for each term on the right as in the case $0<s<1$.	

It remains to verify (ii).  Given $\psi \in H^{-s}(T)$ with $s>\frac{1}{2}$, we obtain
\[
 |\psi|_{-s , T}
 =
 \sup_{\phi \in C^{\infty}_0 (T)} 
  \frac{\langle \psi ,\phi \rangle}{|\phi|_{s , T}}  
 \simeq
 \sup_{\hat{\phi} \in C^{\infty}_0 (\That)} 
	\frac{ |\det(A)| \langle \psihat, \phihat\rangle}
	{ h_T^{1-s} |\phihat|_{s, \That}}
 \simeq
 h_T^{1+s} |\psihat|_{-s,\That}
\]
with the help of (i).	
\end{proof}

\begin{lemma}[A fractional Bramble--Hilbert inequality]
\label{L:Poincare-bis}
Let $1<s<2$.  There is constant $C_{BH}$ depending only on $s$ such that, if $\phi\in H^s(\Hat{T})$ vanishes at the vertices of $\Hat{T}$, then
\[
 \| \phi \|_{s,\hat T}
 \le
 C_{BH} | \phi |_{s, \hat T}.
\]
\end{lemma}

\begin{proof}
We let $\Hat{I}_1$ denote the Lagrange interpolation operator onto the polynomials $\P^1$ of degree at most 1.  Given any polynomial $p\in\P^1$ of degree at most 1, we may write
\[
 \phi
 =
 \phi - \Hat{I}_1\phi
 =
 (\phi-p) + \Hat{I}_1(\phi- p).
\]
Since $\| \Hat{I}_1\psi \|_{1,\Hat{T}} \leq C \max_{\Hat{T}} |\psi|$ for some constant $C$, Lemma~\ref{L:basic-properties}~(\ref{it:sobolev-embedding}) yields
\[
 \| \phi \|_{1,\Hat{T}}
 \le
 C \| \phi- p\|_{s,\Hat{T}},
\]
where $C$ additionally depends on $s>1$.  We choose $p\in\mathcal{P}^1$ such that $\int_{\Hat{T}} p = \int_{\Hat{T}} \phi$ and $\int_{\Hat{T}} \partial_i p = \int_{\Hat{T}} \partial_i \phi$ for $i=1,2$.  Thus, we can conclude with the help of the classical Poincar\'e inequality and its counterpart Lemma \ref{L:Poincare-meanvalue}. 
\end{proof}

\subsection{Continuous and dual problem}
%
%
We assume that $\Omega\subset\R^2$ is a two-di\-men\-sion\-al polygonal, but not necessarily convex domain with Lipschitz boundary $\partial\Omega$.  For any $\theta>0$, Lemma~\ref{L:basic-properties}~(\ref{it:sobolev-embedding}) then implies
that $H^{1+\theta} $ is continuously embedded in $C^0(\ol\Omega)$ and, therefore, the right-hand side of Poisson's equation in \eqref{Problem} belongs to the space $H^{-1-\theta}$.

Writing $s=1-\theta$, we are thus led to consider the Dirichlet problem
\begin{equation}
\label{Poisson}
 -\Delta u = f \quad \text{in } \Omega,
\qquad
  u = 0 \quad \text{on } \partial\Omega,
\end{equation}
where $f \in H^{s-2}$ and $s<1$.  Expecting that the solution is in $H^s$, we additionally require $s>\frac{1}{2}$ to provide a meaning to the boundary condition in \eqref{Poisson} by means of  Lemma~\ref{L:basic-properties}~\eqref{it:trace}.  Furthermore, since we shall invoke duality arguments, it will be useful to consider also the range $1<s<\frac{3}{2}$.  In any case, the differential operator $-\Delta$ should be understood in the distributional sense:
\[
 \scp{ -\Delta u}{ \varphi}
 :=
 \scp{ u }{ -\Delta \varphi } ,
\qquad \forall \varphi \in C_0^\infty(\Omega).
\]
Integration by parts shows that, if $v\in C_0^\infty(\Omega)$, then
\begin{equation}
\label{org-B}
  \scp{ v }{ -\Delta \varphi }
  =
 \B [v, \varphi]
 :=
 \int_\Omega \nabla v \cdot \nabla \varphi \ dx.
\end{equation}
Ne\v{c}as \cite{Necas} extended $\B$ to a continuous bilinear form on $H^s_0\times H^{2-s}_0$ by means of the following lemma. 
\begin{lemma}[Weak derivative and fractional spaces]
\label{L:Necas-4.1}
Let $0 \le \theta < 1/2$ and $G$ be a bounded Lipschitz domain of $\R^d$.  Then there exists a constant $C$ such that
\[
 \int_G \frac{\partial f}{\partial x_i} g \, dx 
 \le
 C \| f \|_{1-\theta,G} \| g \|_{\theta,G},
\qquad\text{for all } f,g \in C^\infty(G).
\]
\end{lemma}
Noteworthy, the case $\theta=1/2$ has to be excluded in view of Grisvard \cite[Proposition 1.4.4.8]{Grisvard:85}. Moreover, \cite{Necas} verified that the extension of $\B$ on $H^s_0\times H^{2-s}_0$ satisfies an inf-sup condition and thus obtained the following theorem.
\begin{theorem}[Weak formulation in $H^s_0$]
\label{T:continuous}
Let $\frac{1}{2}<s<\frac{3}{2}$. For any $f \in H^{s-2}$, there exists a unique weak solution $u\in H^s_0$ of the Dirichlet problem~\eqref{Poisson}:
\[
 \B[u,\vphi] = \scp{f}{\vphi},
\qquad
 \forall\vphi\in C_0^\infty(\Omega).
\]
It satisfies the a~priori estimate
\begin{equation*}
 \norma{u}_{s} \le C \norma{f}_{s-2},
\end{equation*}
where the constant $C$ depends only on $s$ and $\Omega$.
\end{theorem}

Theorem \ref{T:continuous} may be used to establish the well-posedness of Problem \eqref{Problem} in suitable fractional Sobolev spaces.

\begin{corollary}[Well-posedness for point sources]
\label{C:Problem}
Problem	\eqref{Problem} has a unique solution $u$ which, for every $\theta>0$, satisfies $u\in H^{1-\theta}_0$ and
\begin{equation*}
 \B[u,\phi] = \sum_{j=1}^N \alpha_j \phi(x_j),
\qquad
 \forall \phi \in H_0^{1+\theta}.
\end{equation*}
\end{corollary}
%
%
Let us fix $0 < \theta < \frac{1}{2}$.  In view of the extension of the bilinear form $\B$ to $H^{1-\theta}\times H^{1+\theta}$ and Corollary \ref{C:Problem},  we refer to
\begin{equation}
\label{Primal}
\begin{aligned}
 \text{given }f \in H^{-1-\theta},
 &\text{ find $u \in H_0^{1-\theta}$ such that }
\\
 \B[u,\phi] &=\scp{f}{\phi}_{-1-\theta , 1+\theta},
\quad
 \forall \phi \in H^{1+\theta}_{0}(\Omega)
\end{aligned}
\end{equation}
as the direct (primal) problem and to
\begin{equation}
\label{Dual}
\begin{aligned}
 \text{given }g \in H^{-1+\theta},
 &\text{ find $w \in H_0^{1+\theta}$ such that}
\\
 \B[\varphi , w] &= \scp{\varphi}{g}_{1-\theta , -1+\theta},
\quad
 \forall \varphi \in H^{1-\theta}_{0}(\Omega),
\end{aligned}
\end{equation}
as the adjoint (dual) problem.

\subsection{Finite element discretization}
%
%
Let $\T$ be a conforming (edge-to-edge) triangulation of $\Omega$.  We refer to the minimum angle appearing in $\T$ as shape coefficient $\sigma_\T$.   Moreover, we denote by $\E$ the set of all its edges and by $\V$ the set of all its vertices.  The star around a vertex $z \in \V$ is given by
\[
 \omega_z
 :=
 \bigcup_{ T \in \T :  z  \in T } T.
\]

Given $\ell \in \mathbb N$, we let $\mathcal{P}^{\ell}$ denote the space of polynomials of total degree $\le\ell$.  Moreover, let $\VT $ be the finite element space of continuous piecewise polynomials that vanish at the boundary, i.e.
\[
 \VT
 =
 \VT^\ell
 =
 \{ V \in C(\overline{\Omega}) :
     V=0 \text{ on } \partial \Omega; \;
     V|_T \in \mathcal{P}^{\ell} , \forall T \in \T
 \}.
\]
The set of the standard nodes (the locations of the degrees of freedom) of $\VT$ is indicated by $\N$.  We thus have $\V\cap\Omega\subset\N$, with equality for $\ell=1$.

The approximate solution $U$ of \eqref{Problem} is defined as the Galerkin solution in $\VT$:
\begin{equation}\label{galerkin}
   U \in \VT : \quad \B[ U , V] = \sum_{j=1}^N \alpha_j V(x_j) 
\quad
  \forall V \in \VT.
\end{equation}
Notice that, although Problem \eqref{Problem} is associated with a non-symmetric weak formulation, $U$ is computed by solving the usual symmetric, positive definite linear system.


%
%
%
\section{A~posteriori error analysis}
\label{S:apost-analysis}
%
%
In this section, we derive a posteriori bounds for the error $|u-U|_{1-\theta}$, where $0<\theta<1/2$ and $|\cdot|_{1-\theta}$ is a norm thanks to Lemma \ref{L:Poincare}.  Before embarking on their derivation, we define estimator and data oscillation.

\subsection{Estimator and data oscillation}
%
%
We start by defining the error estimator.  To this end, we denote by $h_T=|T|^{\frac12}$ the local meshsize and let $E\in\E$ be an edge.  If $E$ is an interelement edge, we write $E=T_1 \cap T_2$ with $T_1$, $T_2\in\T$ and define the jump 
$ \llbracket \nabla U \rrbracket_{|E} $ of the flux by
\[
 \llbracket \nabla U \rrbracket_{|E}
 :=
 \nabla U^1 \cdot n^1 + \nabla U^2 \cdot n^2,
\]
where $ U^1 $, $ U^2 $ denote the restrictions of $U$ to $ T_1 $, $ T_2 $, respectively, and $ n^1$, $ n^2 $ are the outer normals of $ T_1 $, $ T_2 $.  If $E$ is a boundary edge, we have $E\subset\partial\Omega$ and set $ \llbracket \nabla U \rrbracket_{|E} := 0$.  With these notations,
we define estimator and indicators by
\begin{equation}
\label{est}
 \eta_\theta^2
 =
 \sum_{T\in\T} \eta_{T , \theta}^2
\quad\text{and}\quad
\eta_{T , \theta}^2
=
h_T^{2+2\theta} \norma{\Delta U}_{L^2 (T )}^2
+
h_T^{1+2\theta}
\norma{\llbracket \nabla U\rrbracket}_{L^2 (\partial T)}^2.
\end{equation}
Notice that $\eta_{T,\theta}$ depends only on the approximate solution $U$ and the local meshsize $h_T$ and, thus, is independent of the point sources in Problem \eqref{Problem}. 

\medskip The oscillation is tailored to the specific class of source term in Problem \eqref{Problem}.  It is defined starwise and depends on the interplay of the boundary, the nodes and the supports of the point sources.  We write
\[
 \hat{\mathcal{N}}_\ell=\mathcal{N}_\ell \cup \partial \Omega
\]
for short. For each vertex $z \in \V$, we consider only point sources whose supports are not nodes and collect them according to their sign:
\begin{equation*}
 A_z^{+}
 :=
 \{ j : x_j \in \omega_z \setminus  \mathcal{N}_{\ell} 
\text{ and } \alpha_j >0 \},
\quad
 A_z^{-}
 :=
 \{ j : x_j \in \omega_z \setminus \mathcal{N}_{\ell} \text{ and } \alpha_j <0 \}.
\end{equation*}
  Given $j\in A_z^{+}$, we set
\begin{equation*}
 \sigma_{z,j}^{+}
 :=
 \min \bigg\{
  \dist(x_j, \hat{\mathcal{N}}_\ell)^{\theta}
   +  \max_{i \in A_z^{-}} \dist(x_i, \hat{\mathcal{N}}_\ell)^{\theta} \, ,
 \, \max_{i \in A_z^{-}} |x_j - x_i|^{\theta}
\bigg\} , 
\end{equation*}
if $A_z^{-}\neq \emptyset$, and $\sigma_{z,j}^{+}=0$ otherwise. Similarly, given $j \in A_z^{-} $
\begin{equation*}
 \sigma_{z,j}^{-}
 :=
 \min \bigg\{
  \dist(x_j, \hat{\mathcal{N}}_\ell)^{\theta}
   +  \max_{i \in A_z^{+}} \dist(x_i, \hat{\mathcal{N}}_\ell)^{\theta} \, ,
   \, \max_{i \in A_z^{+}} |x_j - x_i|^{\theta}
   \bigg\},
\end{equation*}
if $A_z^{+}\neq \emptyset$, and $\sigma_{z,j}^{-}=0$ otherwise.  Moreover, let $\lambda_z$ denote the piecewise affine function that is 1 at $z$ and $0$ in all other vertices of the star $\omega_z$.  The oscillation indicator associated to $z$ is
\[
 \xi_{\theta}(z)
 :=
 \begin{cases}
  \sum_{j=1}^{N}
   \dist(x_j, \hat{\mathcal{N}}_{\ell})^{\theta} |\alpha_j| \lambda_z(x_j)
  &\text{if } z\in \partial \Omega,
 \\ 
  \min \big\{
   \sum_{j\in A_z^+} \sigma_{z,j}^+ |\alpha_j|\lambda_z(x_j),
   \sum_{j\in A_z^-} \sigma_{z,j}^- |\alpha_j| \lambda_z(x_j)
   \big\}
& \text{if } z\in \Omega,
\end{cases}
\]
with the convention $\sum_\emptyset=0$. The global oscillation term is then
\begin{equation}
\label{osc}
 \xi_\theta
 := \bigg(
  \sum_{z \in \V} \xi_{\theta}^2 (z)
 \bigg)^{\frac{1}{2}}.
\end{equation}
Let us conclude this section with two remarks, which highlight useful properties of this oscillation.

\begin{remark}[Scaling of oscillation]
\label{R:osc-scaling}
The error estimator $\eta_\theta$ involves, as scaling factors, suitable powers of the local meshsizes $h_T$, $T\in\T$.  For example, the $L^2$-norm of element residual $\Delta U$ is scaled by $h_T^{1+\theta}$.  In order to derive a bound for $\xi_\theta$ with similar scaling factors, we observe that, for any vertex $z$ of a triangle $T\in\T$, we have
\[
 \sigma_{z,j}^{\pm} \lesssim h_T^\theta,
\]
where the hidden constant depends on the shape coefficient $\sigma_\T$.  Since the triangles of the patch $\omega_T = \bigcup\{T'\in\T : T'\cap T \neq \emptyset\}$ have similar areas, this readily leads to
\begin{equation}
\label{alt_osc}
 \xi_{\theta}
 \lesssim
 \sum_{T\in \T} h_T^{\theta} \left(
  \sum_{x_j\in \omega_{T}}|\alpha_j|\Upsilon(T)
 \right)
\end{equation}
with
\[
 \Upsilon(T)
 =
 \begin{cases}
	0
	&\text{if } T\cap \partial\Omega=\emptyset \ \text{and}\  \alpha_i \alpha_j >0,\ \forall x_i,x_j\in \omega_T,
	\\
	1 &\text{otherwise}.
 \end{cases}
\]
The scaling factor of the oscillation is thus given by $h_T^\theta$.  In the case of quasi-uniform refinement,  the global counterpart of this scaling factor corresponds to the decay rate of the error $|u-U|_{1-\theta}$ under consideration; see \cite{Scott}.  The bound in \eqref{alt_osc} may be also used as a triangle-indexed alternative for \eqref{osc}, which however overestimates whenever the distances between point sources are much smaller that the local meshsize.
\end{remark}

\begin{remark}[Oscillation and refinement]
\label{R:osc-ref}
The local oscillation $\xi_\theta(z)$, $z\in\V$, vanishes in many cases and, in any event, asymptotically.  To see this, we observe that $\xi_{\theta}(z) \neq 0$ implies, according to the location of $z$, one of the following conditions:
\begin{enumerate}[\ $\bullet$]
\item If $z\in\partial\Omega$ is a boundary vertex, there is a point source located in $\omega_z\setminus\hat{\mathcal{N}}_{\ell}$.
\item If $z\in\Omega$ is a interior vertex, there are point sources with different sign located in $\omega_z\setminus\hat{\mathcal{N}}_{\ell}$.
\end{enumerate}
Remarkably, if one of these condition is verified, a finite number of refinements step ensures that its negation is met and this remains so for further refinements.  Hence, $\xi_\theta=0$ can always be reached after a finite number of suitable refinement steps.
\end{remark}



%
%
\subsection{Upper Bound}
\label{S:upper-bound}
The estimator \eqref{est} and the oscillation \eqref{osc} provide an upper bound for the error in $H^{1-\theta}$.

\begin{theorem}[Upper bound]
\label{T:upper-bound}
Let $u$ be the solution of Problem \eqref{Problem} and $U$ its approximation associated with the triangulation $\T$. There exists a constant $C_U$, depending on $\Omega$, $\theta\in (0,\frac{1}{2})$ and the shape coefficient $\sigma_\T$ of $\T$, such that
\[
 \left|{u-U}\right|_{1-\theta}
 \le
 C_U ( \eta_\theta + \xi_\theta ).
\]
\end{theorem}

It is worth observing that this upper bound will simplify under adaptive refinement.

\begin{remark}[Asymptotic form of upper bound]
\label{R:asym-ubd}
Point sources generate singularities which are centered at their supports.  Since an adaptive algorithm will refine around these places, Remark \ref{R:osc-ref} suggests that, after a finite number of adaptive refinements, the upper bound of Theorem \ref{T:upper-bound} becomes
\[
 \left|{u-U}\right|_{1-\theta}
 \le
 C_U \eta_\theta.
\]
This expectation is in line with our numerical experiments in \S\ref{S:experiments}.  Interestingly, the asymptotic form is independent of the point sources in Problem \eqref{Problem}.
\end{remark}

We now prove Theorem \ref{T:upper-bound}, postponing some technical estimates about the interpolation error to Lemma \ref{L:interp} below. 

\begin{proof}[Proof of Theorem \ref{T:upper-bound}]
We split the proof in several steps.

\step{1}  We start by relating the error with a suitable norm of the residual.  To this end, we use a duality argument relying on
\begin{equation*}
 |u-U|_{1-\theta}
 =
 \sup_{g \in H^{-1+\theta}}
  \dfrac{\scp{u-U}{g}}{\norma{g}_{-1+\theta}},
\end{equation*}
which is a special case of \eqref{norm-duality}.  Let $g \in H^{-1+\theta}$ and denote by $w$ the solution to the dual problem \eqref{Dual}.  Theorem~\ref{T:continuous} and the original symmetry of $\B$ in \eqref{org-B} yield that $w \in H^{1+\theta}_0$ with $\left | w \right |_{1+\theta} \le C_{\theta,\Omega}^* \left\| g \right\|_{-1+\theta} $.  Since
\begin{equation*}
 \scp{\e }{ g }
 =
 \B[ \e , w ]
 =
 \scp{R}{w}
 :=
 \sum_{j=1}^N \alpha_j w(x_j) - \B[U,w],
\end{equation*}
we obtain
\begin{equation}
\label{error<}
 |u-U|_{1-\theta}
 \le
 C_{\theta,\Omega}^* \| R \|_{-1-\theta}.
\end{equation}
It thus remains to bound the residual norm $\| R \|_{-1-\theta}$.

\step{2} We rewrite $\scp{R}{w}$ for a fixed $w\in H^{1+\theta}$ by means of the partition of unity $\sum_{z\in\V} \lambda_z = 1$ and by exploiting $\scp{R}{\lambda_z}=0$ for all interior vertices $z\in\V\cap\Omega$.  To this end, we let $I_{\T}w$ denote the Lagrange interpolant of $w$ onto $\VT^\ell$ and set $\tw := w - I_{\T} w$.  Moreover, if $z\in\V\cap\Omega$ is an interior vertex, we let $c_z\in\R$ to be chosen later, while, if $z\in\V\cap\Omega$ is boundary vertex, we set $c_z := 0$.  Then
\begin{equation}
\label{R<}
  \scp{R}{w}
  =
  \scp{R}{w - I_{\T} w}
  =
  \sum_{z\in\V} \scp{R}{(\tw-c_z)\lambda_z}
\end{equation}
with the local contributions
\begin{equation}
\label{loc-residual}
 \scp{R}{(\tw-c_z)\lambda_z}
 =
 \sum_{j\in A_z} \alpha_j \big[ \tw(x_j)-c_z \big]\lambda_z(x_j)
 \; - \;
 \B[U,(\tw-c_z)\lambda_z]
\end{equation}
and $A_z=A_z^+ \cup A_z^-$.

\step{3} Fix any $z\in\V$.  Assuming that $c_z = \tw(x^z)$ for some $x^z\in\omega_z$, we bound the second term in \eqref{loc-residual} as follows:
\[
 |\B[  U , (\tw-c_z)\lambda_z]|
 \lesssim
 \bigg(
  \sum_{T \in \Tz}
   h_T^{1+\theta} \norma{\Delta U}_{0,T}
  +
  \sum_{E \in \Ez}
   h_E^{\frac12+\theta} \norma{\llbracket\grad U\rrbracket}_{0,E}
 \bigg) \left|w\right|_{1+\theta,\omega_z},
\]
where $h_E$ indicates the length of an edge $E$ and $\Tz := \{ T \in \T : z\in T \}$ and  $\Ez := \{ E \in \E : z \in E\}$ stand for the triangles and interior edges of the star $\omega_z$, respectively.  To this end, we integrate by parts on each $T\in\Tz$, use $\lambda_z\le1$ and obtain
\begin{equation*}
 |\B[  U , (\tw-c_z)\lambda_z]|
 \le
 \sum_{T \in \Tz} \norma{\Delta U}_{0,T}
  \norma{\tw-c_z}_{0,T}
 +
 \sum_{E \in \Ez} \norma{\llbracket\grad U\rrbracket}_{0,E}
  \norma{\tw-c_z}_{0,E}.
\end{equation*}
Adopting standard arguments to the setting at hand, see Lemma~\ref{L:interp} below, yields
\[
 \left\| \tw - c_z \right\|_{0,T}
 \lesssim
 h_T^{1+\theta} |w|_{1+\theta,\omega_z}
\quad\text{and}\quad
 \left\| \tw - c_z \right\|_{0,E}
 \lesssim
  h_E^{\frac12 +\theta} |w|_{1+\theta,\omega_z},
\]
where the hidden constants depend on $\theta$ and the shape coefficient $\sigma_\T$.  Inserting this in the preceding inequality, we arrive at the claimed bound for $|\B[U,(\tw-c_z)\lambda_z]|$.
 
\step{4} Let $z\in\V$.  For given or appropriately chosen $c_z$, we derive the following bound for the sum over $A_z$ in \eqref{loc-residual}:
\begin{equation}
\label{xi-z}
 \bigg|
  \sum_{j \in A_z} \alpha_j \big[\tw(x_j)-c_z \big] \lambda_z(x_j)
 \bigg|
 \lesssim
 \xi_\theta(z) \left|w\right|_{1+\theta,\omega_z}.
\end{equation}
The possible choices of $c_z$ depend on the location of $z$ as well as the point sources located in $\omega_z$.

\smallskip\noindent\emph{Case 1:} $z\in\Omega$ and $A_z=\emptyset$.
Then $\sum_{j\in A_z} \alpha_j \big[ \tw(x_j)-c_z \big] \lambda_z(x_j) = 0$, irrespective of the choice of $c_z$.  In particular, \eqref{xi-z} is verified and we may take $c_z = 0$. 

\smallskip\noindent\emph{Case 2:} $z\in\Omega$ and $A_z\neq\emptyset$.  Then we have $A_z^-\neq\emptyset$ or $A_z^+\neq\emptyset$.  If the latter occurs, then $\sum_{k \in A_z^+} \alpha_k  \lambda_z(x_k)>0$ and we can consider
\begin{equation}
\label{cz+}
 c_z^+
 =
 \sum_{j \in A_z^+} \beta^+_j \tw(x_j)
\quad\text{with}\quad
 \beta^+_j
 :=
 \frac{\alpha_j \lambda_z(x_j)}
  {\sum_{k \in A_z^+} \alpha_k  \lambda_z(x_k)} \in (0,1],
\end{equation}
which implies
\begin{equation}
\label{Az;-}
 \sum_{j \in A_z} \alpha_j \big[ \tw(x_j)-c_z^+ \big] \lambda_z(x_j)
 =
 \sum_{j \in A_z^-}
  \alpha_j \big[ \tw(x_j)-c_z^+ \big] \lambda_z(x_j).
\end{equation}
Fix $j \in A_z^-$ for a moment.  On the one hand, the definition of $c_z^+$ and Lemma~\ref{L:interp} below yield 
\begin{equation}
\label{sigma-1}
 | \tw(x_j) - c_z^+ |  
 \le
 \sum_{i \in A_z^+} \beta^+_i | \tw(x_j) - \tw(x_i) |
 \lesssim
 \max_{i \in A_z^{+}} |x_j - x_i|^{\theta}
  | w |_{1+\theta,\omega_z}.
\end{equation}
On the other hand, since $\tw(\nu)=0 $ for all $\nu \in \hat{\mathcal{N}}_\ell$, we have that, for any choice of $\nu_i \in \hat{\mathcal{N}}_\ell$, $i \in \{0\}\cup A_z^+$,
\[
 |\tw(x_j) - c_z^+ |
 \le
 |\tw(x_j) - \tw(\nu_0) |
 +
 \sum_{i \in A_z^+} \beta^+_i | \tw(x_i) - \tw(\nu_i)|,
\]
which again by Lemma~\ref{L:interp} implies that
\begin{equation}\label{sigma-2}
|\tw(x_j) - c_z^+ |
\lesssim \left( \dist(x_j, \hat{\mathcal{N}}_\ell)^{\theta} + \max_{i \in A_z^{+}}
\dist(x_i, \hat{\mathcal{N}}_\ell)^{\theta} \right) | w |_{1+\theta,\omega_z} .
\end{equation}
Combining the bounds \eqref{sigma-1} and \eqref{sigma-2} with the definition of $\sigma_j^-$ gives
\[
 |\tw(x_j) - c_z^+|
 \lesssim
 \sigma_j^-| w |_{1+\theta,\omega_z},
\quad \text{for all } j \in A_z^-,
\]
and thus, upon recalling \eqref{Az;-},
\begin{equation}\label{xi-1}
 \bigg|
  \sum_{j \in A_z} \alpha_j \big[ \tw(x_j)-c_z^+ \big] \lambda_z(x_j)
 \bigg|
 \lesssim
 \sum_{j \in A_z^-}
  |\alpha_j| \sigma_{j}^- \lambda_z (x_j)
  |w|_{1+\theta,\omega_z}.
\end{equation}

Moreover, if we have $A_z^-\neq\emptyset$, we can consider
\[
 c_z^-
 =
 \sum_{j \in A_z^-} \beta^-_j \tw(x_j)
\quad\text{with}\quad
 \beta^-_j
 :=
 \frac{ -\alpha_j \lambda_z(x_j)}
  {\sum_{k \in A_z^-} (-\alpha_k)  \lambda_z(x_k)}
 > 0
\]
Notice that the definition of $c_z^-$ is the one of $c_z^+$, if we replace $A_z^-$ by $A_z^+$.  Consequently, we can argue as before and obtain here
\begin{equation}
\label{xi-2}
 \bigg|
  \sum_{j \in A_z} \alpha_j \big[ \tw(x_j)-c_z^- \big] \lambda_z(x_j)
 \bigg|
 \lesssim
 \sum_{j \in A_z^+}
  |\alpha_j| \sigma_{j}^+ \lambda_z (x_j)
  \left|w\right|_{1+\theta,\omega_z}.
\end{equation}
Taking the minimum of the two bounds \eqref{xi-1} and \eqref{xi-2} verifies \eqref{xi-z} in this case.

\smallskip\noindent\emph{Case 3:} $z\in\partial\Omega$.  Here we have $c_z = 0$. Using Lemma~\ref{L:interp} and $\tw(x) = 0$ for all $x \in \hat{\mathcal{N}}_\ell$, we derive
\begin{align*}
 \sum_{j \in A_z} \alpha_j \big[ \tw(x_j)-c_z \big] \lambda_z(x_j) 
 &=
 \sum_{j \in A_z} \alpha_j \tw(x_j)\lambda_z (x_j)
\\
 &\lesssim   
 \left(
  \sum_{j \in A_z} |\alpha_j|
   \dist(x_j,\hat{\mathcal{N}}_\ell)^\theta
 \right) 
 |w|_{1+\theta,\omega_z},
\end{align*}
which  verifies \eqref{xi-z} also in this case.

\smallskip\noindent Notice that all choices of $c_z$ in Cases 1--3 satisfy $\min_{\omega_z} \tw \leq c_z \leq \max_{\omega_z} \tw $.  Since $\tw$ is continuous and $\omega_z$ compact, we can always choose $x^z\in\omega_z$ such that $\tw(x^z)=c_z$ and therefore apply Step 3. 

\step{5} Combining Steps 3 and 4 and using $h_E\lesssim h_T$ whenever $E\subset T$, we derive
\begin{equation*}
 \scp{R}{(\tw-c_z)\lambda_z}
 \lesssim
 \left[
  \sum_{T \in \Tz} \eta_{T,\theta}
  +
  \xi_\theta(z)
 \right] |w|_{1+\theta,\omega_z}.
\end{equation*}
We insert this inequality into \eqref{R<}, use the Cauchy--Schwarz inequality for sums twice, observe that the cardinality of $\Tz$ is bounded in terms of $\sigma_\T$ and arrive at
\begin{equation*}
 \scp{R}{w}
 \lesssim
 \left(
  \eta_\theta^2 + \xi_\theta^2
 \right)^{1/2}
 \left(
  \sum_{z \in \V} |w|_{1+\theta,\omega_z}^2
 \right)^{1/2}.
\end{equation*}
Subsequently, $\sum_{z \in \V} \left|w\right|_{1+\theta,\omega_z}^2 \le 3 |w|_{1+\theta,\Omega}^2$ and \eqref{error<} finish the proof.
\end{proof}

We turn to the postponed estimates about the interpolation error.

\begin{lemma}[Interpolation error]
\label{L:interp}
Let $ w \in H^{1+\theta} $ with $0 < \theta <\frac12$ and consider $\tw := w - I_{\T} w$, where $I_{\T} w$ denotes the Lagrange interpolant of $w$ into $\VT^\ell$. Moreover, let $\omega_z$, $z\in\V$, be any star of $\T$.  Given any $x,y\in \omega_z$, we have
\[
|\tw(x)-\tw(y)|
\lesssim
|x - y|^{\theta} |w|_{1+\theta,\omega_z},
\]
and, if $c_z = \tw(x^z)$ for some $x^z \in \omega_z$, then
\[
  \left\| \tw - c_z \right\|_{0,T}
  \lesssim h_T^{1+\theta} |w|_{1+\theta,\omega_z},
\quad
 \left\| \tw - c_z \right\|_{0,E}
 \lesssim
 h_E^{\frac12 +\theta} |w|_{1+\theta,\omega_z}
\]
for any triangle $T\in\T$ and any edge $E\in\E$ containing $z$.  The hidden constants depend only on $\Omega$, $\theta$, $\ell$ and $\sigma_\T$, while $h_E$ stands for the length of $E$.
\end{lemma}

\begin{proof}
We start by deriving the first inequality, where the domain is the reference triangle $\That$ instead of some star $\omega_z$.  Thanks to Lemma \ref{L:basic-properties}~(\ref{it:sobolev-embedding}), the bound $\|I^\ell w\|_{1+\theta,\That} \lesssim \max_{\That} |w|$, and Lemma \ref{L:Poincare-bis}, we have, for $x,y\in\That$,
\begin{equation*}
 \frac{|\tw(x)-\tw(y)|}{|x-y|^{\theta}}
 \lesssim
 \|\tw\|_{1+\theta,\That}
 \lesssim
 |\tw|_{1+\theta,\That}
  = 
 |w|_{1+\theta,\That},
\end{equation*}
where the hidden constant depends only on $\theta$ and $\ell$.  In view of Lemma~\ref{L:scaling} (i) and its proof, both sides scale in the same way under affine transformations of the domain.  Hence, for any triangle $T\in\T$, we obtain
\begin{equation}
\label{Hoelder;T}
 |\tw(x)-\tw(y)|
 \lesssim
 |x - y|^{\theta} |w|_{1+\theta,T}.
\end{equation}
If $x$ and $y$ are two arbitrary points in $\omega_z$, we connect them with a polygonal path, made of straight segments in each element of $\omega_z$ and having total length $\lesssim |x-y|$.  The existence and construction of such a path is presented in Lemma 3.4 of~\cite{SaccVee}, where the involved constant depends on $\sigma_\T$ and the Lipschitz constant associated with $\partial\Omega$.  Applying \eqref{Hoelder;T} segmentwise, we obtain the first claimed inequality.  Integrating it, we readily deduce the other ones.
%
\end{proof}

\subsection{Lower Bounds}
\label{S:lower-bound}
In this section, we assess the sharpness of the upper bound in Theorem \ref{T:upper-bound}, dealing with the two parts $\eta_\theta$ and $\xi_\theta$ separately.

\medskip Let us start with the oscillation $\xi_\theta$ defined in \eqref{osc}.  We first recall that, typically, oscillation terms are not shown to be bounded by the error, but are, formally, of higher order.  Here, we encounter similar properties for $\xi_\theta$.  Indeed, Remark \ref{R:osc-scaling} suggests that, under global uniform refinement, $\xi_\theta$ decreases at least with the order of the error $|u-U|_{1+\theta}$.  Moreover, Remark \ref{R:osc-ref} suggests that  $\xi_\theta$ even vanishes after a finite number of appropriate refinements in a reasonable adaptive algorithm.  In such cases, $\xi_\theta$ is then of arbitrarily higher order. 

Our main result about the sharpness of $\eta_\theta$ from \eqref{est}, or of the asymptotic form of the upper bound in Remark \ref{R:asym-ubd}, is as follows.
\begin{theorem}[Lower bounds]
\label{T:lower-bound}
Let $u$ be the solution of Problem \eqref{Problem}, $U$ its approximation associated with the triangulation $\T$ and $0<\theta<1/2$.  For any triangle $T\in\T$, we have the local bound
\begin{equation*}
 \eta_{T,\theta}
 \le
 C_L \left|u-U\right|_{1-\theta, \omega_T},
\end{equation*}
where $\omega_T$ is the patch of all triangles of $\T$ sharing a side with $T$.  Furthermore, we have also the global bound
\begin{equation*}
 \eta_{\theta}
 \le
 \tilde C_L |u-U|_{1-\theta} 
\end{equation*}
Both constants $C_L$, $\tilde C_L$ depend only on $\theta$, the polynomial degree $\ell$, the shape coefficient $\sigma_\T$ and the number $N$ of point sources. 
\end{theorem}

\begin{remark}[Asymptotic independence on point sources] 
The dependence on $N$ is actually through the maximum number $N_\T$ of point sources supported in one element of $\T$.  After a finite number of suitable refinement steps, every triangle will contain at most the support of one point source.  For the same reason as in Remark \ref{R:asym-ubd}, one expects that these refinement steps are actually quickly accomplished by a (reasonable) adaptive algorithm.  We therefore may say that the constants $C_L$ and $\tilde C_L$ are asymptotically independent of $N$.  Combining this with the asymptotic upper bound in Remark \ref{R:asym-ubd}, we see that, asymptotically, the error $|u-U|_{1-\theta}$ is encapsulated with a~posteriori quantities that are independent on the point sources in Problem \eqref{Problem}.
\end{remark}

The proof of Theorem \ref{T:lower-bound} uses the constructive approach of Verf\"urth \cite{V}.  In order to adapt it to our setting with fractional Sobolev space at hand, we need the following preparations concerning the local continuity of $\B$ and suitable test functions with local support.  These test functions will be products of polynomials and cut-off functions.  For the latter, we shall use the following type: given a ball of radius $r$ with midpoint $z\in\R^2$, set
\[
 \eta_B (x)
 =
 \eta \left( \frac{x-z}{r} \right)
\quad\text{where}\quad
 \eta(x)
 =
 \begin{cases} 
  \exp \left( \frac{1}{|x|^2-1} \right) &\text{if } |x| < 1, 
 \\
  0 &\text{otherwise}.
\end{cases}
\]

\begin{lemma}[Cut-off within triangles]
\label{L:aux1}
Let $k\in\mathbb{N}_0$ and  $T$ be a triangle.  Moreover, let $\Hat{B}$ be the ball with maximal radius in the reference triangle $\Hat{T}$ and $F:\R^2\to\R^2$ be an affine bijection with $F(\Hat{T})=T$.  Then the cut-off function $\eta_T := \eta_{\Hat{B}}\circ F^{-1}$ satisfies, for all $V \in \P^k(T)$,
\[
 \int_T V^2 \lesssim  \int_T V^2 \eta_T
\quad\text{and}\quad
 | \nabla(V\eta_T) |_{\theta,T}
 \lesssim
 h_T^{-1-\theta} \norma{V}_{0,T}.
\]
The hidden constants depend only on $\theta$, $k$, and the minimal angle of $T$. 
\end{lemma}

\begin{proof}
In view of the transformation rule and Lemma \ref{L:scaling}, the claim is equivalent to
\[
 \int_{\Hat{T}} \Hat{V}^2
 \lesssim
 \int_{\Hat{T}} \Hat{V}^2 \eta_{\Hat{B}}
\quad\text{and}\quad
 | \nabla(\Hat{V} \eta_{\Hat{B}}) |_{\theta,\Hat{T}}
 \lesssim
 \norma{\Hat{V}}_{0,\Hat{T}}
\]
for all $\Hat{V}\in\P^k(\Hat{T})$.  This statement in turn follows from the equivalence of norms on the finite-dimensional spaces $\P^k(\Hat{T})$ and $\P^k(\Hat{T})/\R$.
\end{proof}

\begin{lemma}[Cut-off across edges]
\label{L:aux2}
Let $k\in\mathbb{N}_0$ and $E = T_1 \cap T_2$ be the common edge of two triangles $T_1$ and $T_2$.  For $i=1,2$, denote by $\Hat{T}_i$ the reference triangles with vertices $(0,0), (1,0), (0,(-1)^i)$ and indicate by $\Hat{B}_{12}$ the ball with maximal radius in the reference patch $\Hat{T}_1\cup\Hat{T}_2$.  Moreover, let $F:\R^2\to\R^2$ be a piecewise affine bijection with $F(\Hat{T}_i)=T_i$, $i=1,2$. Then the cut-off function $\eta_E := \eta_{\Hat{B}}\circ F^{-1}$ satisfies, for all $V \in \P^k(E)$ and $i=1,2$,
\[
 \int_E V^2
 \lesssim
 \int_E V^2 \, \eta_E
\quad\text{and}\quad
 | \nabla(\overline{V} \eta_E) |_{\theta,T_i} 
 \lesssim
 h_E ^{-1-\theta} \norma{ \overline{V} \eta_E }_{0,T_i}
 \lesssim
 h_E^{-\frac12-\theta} \norma{V}_{0,E},
\]
where $h_E$ denotes the length of $E$ and $\overline V$ is a suitable extension of $V$.  The hidden constants depend only on $\theta$, $k$, and the minimal angle of $T$. 
\end{lemma}

\begin{proof}
In view of the transformation rule and Lemma \ref{L:scaling}, the claim is equivalent to the following statement associated with the reference edge given by the vertices $(0,0)$, $(1,0)$: for all $\Hat{V}\in\P^k(\Hat{E})$, we have
\[
 \int_{\Hat{E}} \Hat{V}^2
 \lesssim
 \int_{\Hat{E}} \Hat{V}^2 \eta_{\Hat{B}}
\quad\text{and}\quad
 | \nabla(\overline{\Hat{V}} \eta_{\Hat{B}}) |_{\theta,\Hat{T}_i}
 \lesssim
 \norma{ \overline{\Hat{V}} \eta_{\Hat{B}} }_{0,\Hat{T}_i}
 \lesssim
 \norma{\Hat{V}}_{0,\Hat{E}},
\]
where the extension $\overline{\Hat{V}}$ of $\Hat{V}$ is given by $\overline{\Hat{V}}(x_1,x_2) = \Hat{V}(x_1)$.  Again, these inequalities follow from the equivalence of norms on the finite-dimensional spaces $\P^k(\Hat{T}_i)$ and $\P^k(\Hat{T}_i)/\R$.
\end{proof}

\begin{lemma}[Local continuity of $\B$]
\label{L:local-cont}
Given $0 < \theta < 1/2$ and any triangle $T$, we have
\begin{equation*}
	\int_T \nabla v \cdot \nabla \varphi\, dx 
	\lesssim
	| v |_{1-\theta,T} | \nabla\varphi |_{\theta,T}
\end{equation*}
for all $v \in H^{1-\theta}(T)$ and $\varphi\in H^{1+\theta}(T)$ such that $\supp\varphi = \supp\eta_T$ or $\supp\eta_E$ with $\eta_T$, $\eta_E$ from Lemmas \ref{L:aux1} and \ref{L:aux2}. The meaning of the left-hand side is given by continuous extension and	the hidden constant depends on the minimal angle of $T$.
\end{lemma} 

\begin{proof}
In view of Lemma \ref{L:scaling}, both sides of the claimed inequality scale in the same manner under affine transformations.  We therefore can assume that $T=\Tilde{T}$, where $\Tilde{T}$ is one of the reference triangles $\Hat{T}$, $\Hat{T}_1$, $\Hat{T}_2$. Correspondingly, we write $\Tilde{B}$ for $\Hat{B}$ or $\Hat{B}_{12}$. We set $c = |\Tilde{T}|^{-1}\int_{\Tilde{T}} v$ and apply Lemma \ref{L:Necas-4.1} to obtain 
\begin{align*}
	\int_{\Tilde{T}} \nabla v \cdot \nabla \varphi\, dx
	=
	\int_{\Tilde{T}} \nabla (v-c) \cdot \nabla \varphi\, dx
	\lesssim
	\| v - c\|_{1-\theta,\Tilde{T}}
	 \| \nabla\varphi \|_{\theta,\Tilde{T}}.
\end{align*}
Using Lemma \ref{L:Poincare-meanvalue} and replacing $\partial G$ by the 2-dimensional set $\Tilde{T}\setminus\Tilde{B}$ in step 2 of the proof of Lemma \ref{L:Poincare}, we conclude with
\[
	\| v - c\|_{1-\theta,T}
	\lesssim
	| v |_{1-\theta,T}
	\quad\text{and}\quad
	\| \nabla\varphi \|_{\theta,T}
	\lesssim
	| \nabla\varphi |_{\theta,T}.
	\qedhere
\]
\end{proof}

After these preparations, we are ready to prove the claimed lower bounds.

\begin{proof}[Proof of Theorem~\ref{T:lower-bound}]
\step{1} We shall use test functions, whose support does not contain point sources.  To construct them, we shall exploit Lemmas \ref{L:aux1} and \ref{L:aux2} for the following sub-triangles.
Let $N_\T$ be the maximum number of Dirac masses supported in a triangle $T\in\T$ and set
\begin{equation}
\label{DefM}
 M := 2(N_\T+1).
\end{equation}
We divide each edge $E$ of $\T$ into $M$ equal sub-edges and denote by $\SS_E^M$ the set of these sub-edges.  Moreover, for any triangle $T\in\T$,  we join the endpoints of the sub-edges by lines parallel to the edges of $T$ and so divide $T$ into $M^2$ equivalent sub-triangles.  We indicate with $\SS_T^M$ this set of sub-triangles; see also Figure~\ref{F:bubbles}.   The choice \eqref{DefM} of $M$ ensures that
\begin{enumerate}
\item[(a)] for each $T \in \T$, we can choose a sub-triangle $T^* \in \SS_T^M$ that does not contain any point source appearing in \eqref{Problem}, 
\item[(b)] for each edge $E$ of the triangulation $\T$, we can choose a sub-segment $E^* \in \SS_E^M$ such that no point source appearing in \eqref{Problem} is supported in the union of the two sub-triangles adjacent to $E^*$.
\end{enumerate}

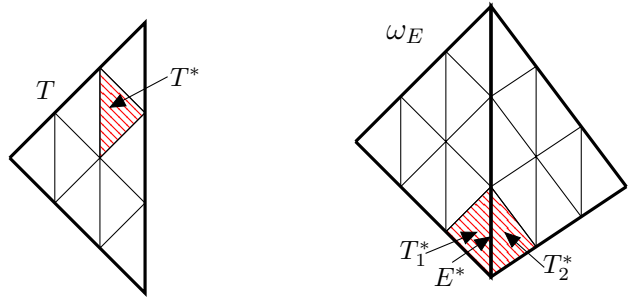
\begin{figure}
	\qquad
	\begin{tikzpicture}[scale=0.6]
	\draw[very thick] (0,0)--(3,3) -- (3,-3)-- (0,0);
	\draw (1,1)--(1,-1)--(3,1)--(2,2)--(2,-2)--(3,-1)--(1,1) ;
	\draw[pattern=north west lines, pattern color=red] (2,2)--(3,1)--(2,0)--(2,2);
	\node at (0.8,1.5) {$T$};
	\node at (3.9,1.8) {$T^*$};
	\draw[-triangle 45] (3.5,1.8)--(2.2,1.11) ;
	\end{tikzpicture}
	\hfil
	\begin{tikzpicture}[scale=0.6]
	\draw[pattern=north west lines, pattern color=red] (2,-2)--(3,-1)--(4,-7/3)--(3,-3)--(2,-2);
	\draw[very thick] (0,0)--(3,3) -- (3,-3)-- (0,0);
	\draw (1,1)--(1,-1)--(3,1)--(2,2)--(2,-2)--(3,-1)--(1,1) ;
	\draw[very thick] (6,-1)--(3,3) -- (3,-3)-- (6,-1);
	\draw (5,1/3)--(5,-5/3)--(3,1)--(4,5/3)--(4,-7/3)--(3,-1)--(5,1/3) ;
	\node at (1.33,-2.48) {$T_1^*$};
	\node at (4.5,-2.79) {$T_2^*$};
	\node at (2.1,-3) {$E^*$};
	\node at (1.11,2.38) {\Large{$\omega_E$}};
	\draw[-triangle 45] (1.55,-2.48)--(2.71,-2) ;
	\draw[-triangle 45] (2.14,-2.68)--(3,-2.1) ;
	\draw[-triangle 45] (4.18,-2.79)--(3.29,-2) ;
	\end{tikzpicture}
	\qquad
	\caption{\emph{Sub-edges, sub-triangles and support of test functions.}}
	\label{F:bubbles}
\end{figure}
Scaling arguments, similar to those in the proofs of Lemmas \ref{L:aux1} and \ref{L:aux2}, yield: if $k\in\mathbb{N}_0$, we have
\begin{equation}
\label{->*}
 \norma{V}_{0,T}
 \lesssim
 \norma{V}_{0,T^*}
\quad\text{and}\quad
  \norma{V}_{0,E}
  \lesssim
  \norma{V}_{0,E^*}
\end{equation} 
for all $V\in\P^k(T)$ or $V\in\P^k(E)$, where the hidden constants depend on $M$, $k$, and the minimal angle in $\T$ but not on the choices of $T^*$ and $E^*$.

\step{2} Let us now prove a lower bound of the local error in terms of any given element residual $h_T^{2+2\theta} \left\|\Delta U\right\|_{0,T}^2$, $T\in\T$.  To this end, we only need to consider $\ell\geq2$ and observe that then $\Delta U\in\P(T)^{\ell-2}$.  Using $\varphi_T =  \Delta U \, \eta_{T^*}$, we derive
\begin{align*}
 \norma{\Delta U}_{0,T}^2
 &\lesssim
 \norma{\Delta U}_{0,T^*}^2
 \lesssim
 \int_{T^*} \Delta U \varphi_T
 =
 -\int_{T^*} \nabla U \cdot \nabla \varphi_T
 =
 \int_{T^*} \nabla (u-U) \cdot \nabla \varphi_T
\\
 &\lesssim
 \left|u-U\right|_{1-\theta,T^*} |\nabla\varphi_T|_{\theta,T^*}
 \lesssim
 h_T^{-1-\theta} \left|u-U\right|_{1-\theta,T}
 \norma{\Delta U}_{0,T}
\end{align*}
with the help of the first inequality in \eqref{->*}, Lemma~\ref{L:aux1}, integration by parts, the choice of $T^*$, Lemma \ref{L:local-cont} and $h_{T^*}\leq h_T\leq M h_{T^*}$. Multiplying by $h_T^{1+\theta} \norma{\Delta U}_{0,T}^{-1}$
and squaring we arrive at
\begin{equation}
\label{int_res_bound}
 h_T^{2+2\theta}  \norma{\Delta U}_{0,T}^2
 \lesssim
 \left|u-U\right|_{1-\theta,T}^2.
\end{equation}

\step{3} Next, we provide a lower bound for the local error in terms of any given jump residual $ \norma{\llbracket\nabla U\rrbracket}_{0,E}^2$, $E$ edge of $\T$.  Let $T_1$, $T_2$ be the two triangles of $\T$ sharing the edge $E$, and let $U^1$, $U^2$, $n^1$, $n^2$ denote the restrictions of $U$ and the outer normals of $T_1$, $T_2$, respectively. Then $\llbracket\nabla U\rrbracket
= \nabla U^1 \cdot n^1 + \nabla U^2 \cdot n^2 \in \P^{\ell-1}(E)$ and denote by $\overline{\llbracket\nabla U\rrbracket}$ its extension from Lemma \ref{L:aux2}.  Using $\varphi_E = \overline{\llbracket\nabla U\rrbracket} \, \eta_{E^*}$, we derive
\begin{align*}
 &\norma{\llbracket\nabla U\rrbracket}_{0,E}^2
 \lesssim
 \norma{\llbracket\nabla U\rrbracket}_{0,E^*}^2
 \lesssim
 \int_{E^*} \llbracket\nabla U\rrbracket \varphi_E
 =
 \sum_{i=1,2} \int_{E^*} \nabla U^i \cdot n^i \, \varphi_E
\\
 &\quad =
 \sum_{i=1,2} \int_{T_i^*} \nabla U \cdot \nabla \varphi_E
                 + \Delta U \, \varphi_E
 =
 \sum_{i=1,2}  \int_{T_i^*} \nabla (U-u) \cdot \nabla \varphi_E
 +    \Delta U \, \varphi_E
\\
 &\quad \lesssim
 \sum_{i=1,2}
  | u-U |_{1-\theta,T_i^*} | \nabla\varphi_E |_{\theta,T_i^*}
   + \norma{\Delta U}_{0,T_i^*} 
      \| \varphi_E \|_{0,T_i^*} \quad
\\
 &\quad \lesssim
 \sum_{i=1,2}
  h_E^{-\frac12 - \theta} 
   \left|u-U\right|_{1-\theta,T_i}
   \| \llbracket\nabla U\rrbracket \|_{0,E}
  +
  h_E^{\frac12}  \norma{\Delta U}_{0,T_i} 
   \| \llbracket\nabla U\rrbracket \|_{0,E}
\end{align*}
with the help of the second inequality in \eqref{->*}, Lemma \ref{L:aux2}, integration by parts, the choice of $E^*$, Lemma \ref{L:local-cont} and $h_{E^*}\leq h_E\leq M h_{E^*}$.  After multiplying by $ h_E^{\frac{1}{2}+\theta} \norma{\llbracket\nabla U\rrbracket}_{0,E}^{-1} $, we obtain
\begin{equation}
\label{jump_res_bound}
  h_E^{\frac{1}{2}+\theta} \norma{\llbracket\nabla U\rrbracket}_{0,E}
  \lesssim
  \sum_{i=1,2}
   \left|u-U\right|_{1-\theta,T_i}
    + h_E^{1+\theta} \norma{\Delta U}_{0,T_i}.
\end{equation}

\step{4} The claimed lower bound in terms of $\eta_T$, $T\in\T$, follows by combining the squares of \eqref{int_res_bound} and \eqref{jump_res_bound} for the involved triangles and interelement edges; recall that $h_E \lesssim h_T$ whenever $E$ is an edge of $T$ and that for boundary edges $E \subset \partial\Omega$, we have set $\llbracket\nabla U\rrbracket_{|E} = 0$.

The global lower bound is a direct consequence of local one: sum the square of all local ones and take into account that the cardinality of $\{T'\in\T \mid \omega_{T'} \supset T \}$ is bounded in terms of the shape coefficient of $\T$.
\end{proof}

\section{Numerical Results}
\label{S:experiments}
%
%
In this section, we numerically test the a~posteriori error estimators of \S\ref{S:apost-analysis}.  To this end, we use it in the adaptive solution of two examples of Problem \eqref{Problem} and analyze resulting properties of the adaptive algorithm.  

The adaptive algorithm, which was implemented within the finite element toolbox ALBERTA~\cite{ALBERTA}, has the following structure.  Given $\theta$ and a conforming initial triangulation $\T_0$ of $\Omega$, it iterates the main steps
\begin{equation}
\label{loop}
\textsc{Solve} \quad \longrightarrow \quad
\textsc{Estimate} \quad \longrightarrow \quad
\textsc{Mark} \quad \longrightarrow \quad
\textsc{Refine}.
\end{equation}
The step \textsc{Solve} consists in solving the discrete system \eqref{galerkin} for the current triangulation $\T$ and linear elements.  The step \textsc{Estimate} then computes the a~posteriori error estimator \eqref{est} and the step \textsc{Mark} selects triangles for refinement by means of the maximum strategy: $T \in \T$ is marked whenever $\eta_{T,\theta} > 0.5 \max_{T' \in \T} \eta_{T',\theta}$. In the step \textsc{Refine}, these marked triangles are bisected twice so that each of their edges is halved.  In doing so, further triangles are bisected in order to maintain the conformity of the next triangulation.
\begin{example}[Fundamental solution]
\label{E:fund-sol}
Consider Problem \eqref{Problem} with data such that
\[
 u(x) = -\frac1{2\pi} \log |x|,
\qquad
 x \in \Omega := (-1,1)^2
\]
is the exact solution, together with the parameter values $\theta = 0$, $0.125$, $0.250$, $0.375$, and $0.5$ for the adaptive algorithm.  Notice that, for $\theta=0$, the error estimator formally corresponds to the infinite error $|u-U|_{1,\Omega}$.  Moreover, $\theta=0.5$ is not covered by the analysis given in \S\ref{S:apost-analysis}, but the convexity of $(-1,1)^2$ suggests that the equivalence of $\eta_{1/2}$ and $|u-U|_{1/2,\Omega}$ is still given.
\end{example}

For given $\theta>0$, the exact solution $u$ formally has almost $\theta$ derivatives (in $L^2$) more than required in the error $|u-U|_{1-\theta,\Omega}$.  Thus, with increasing $\theta$, quasi-uniform meshes ensure increasing error decay, suggesting that the grading of meshes generated with $\eta_\theta$ decreases with increasing $\theta$.  Figure \ref{F:ex1_meshes} confirms this expectation, as well as the corresponding meshes for the intermediate values $\theta = 0.125$, $0.375$, which are not shown.
\begin{figure}[h!tbp]
\begin{center}

\includegraphics[width=.3\textwidth]{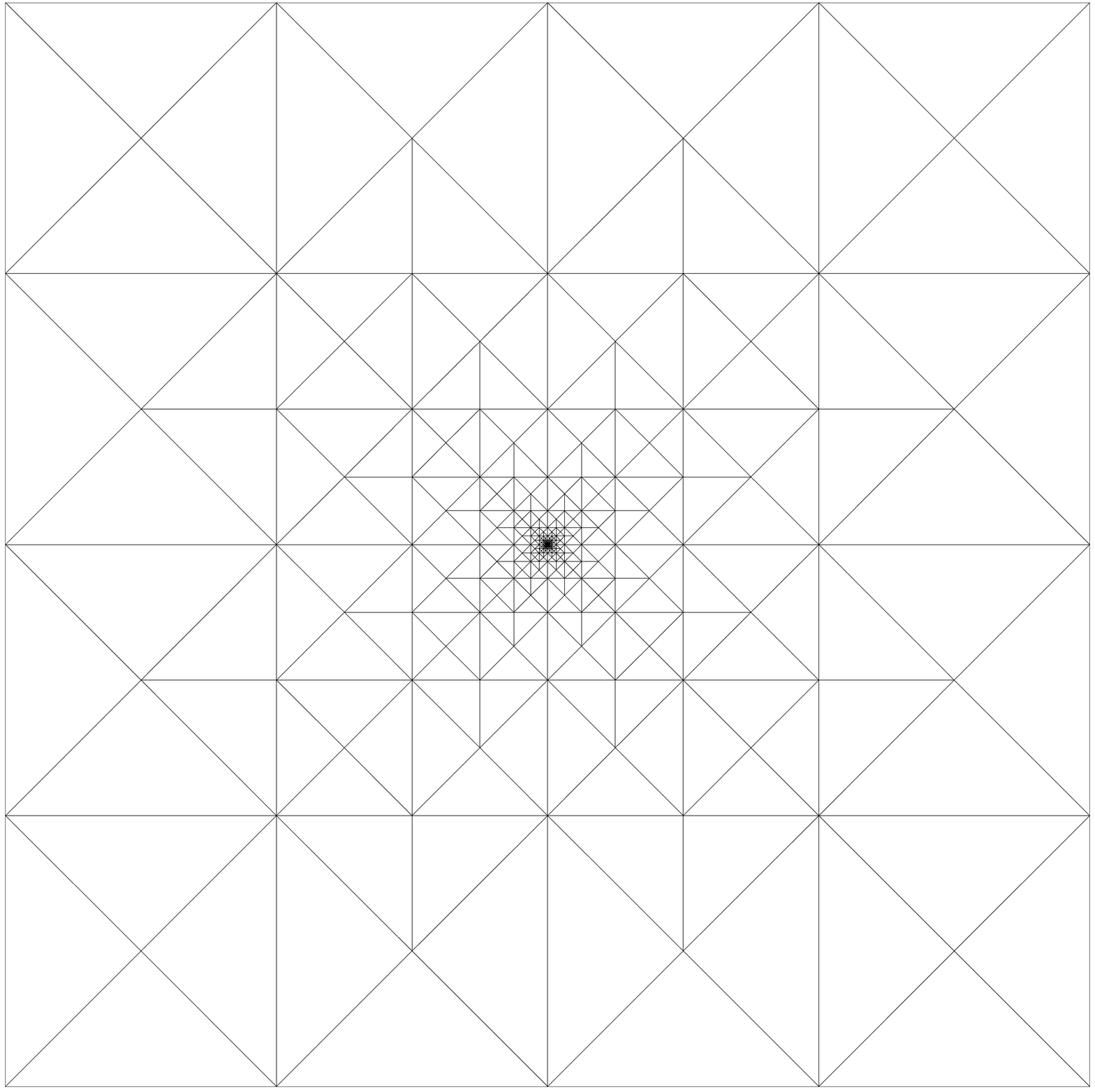}
\includegraphics[width=.3\textwidth]{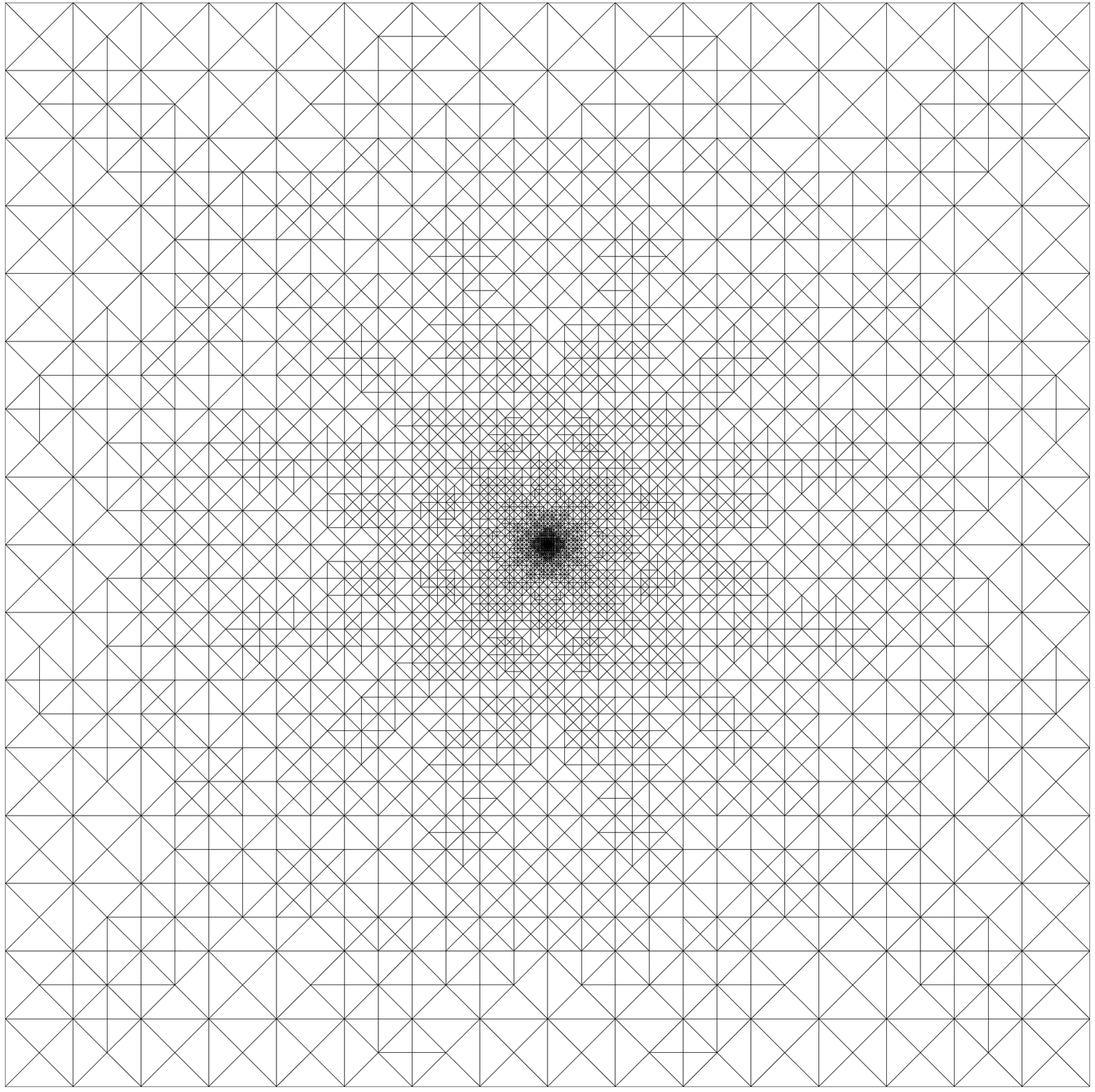}
\includegraphics[width=.3\textwidth]{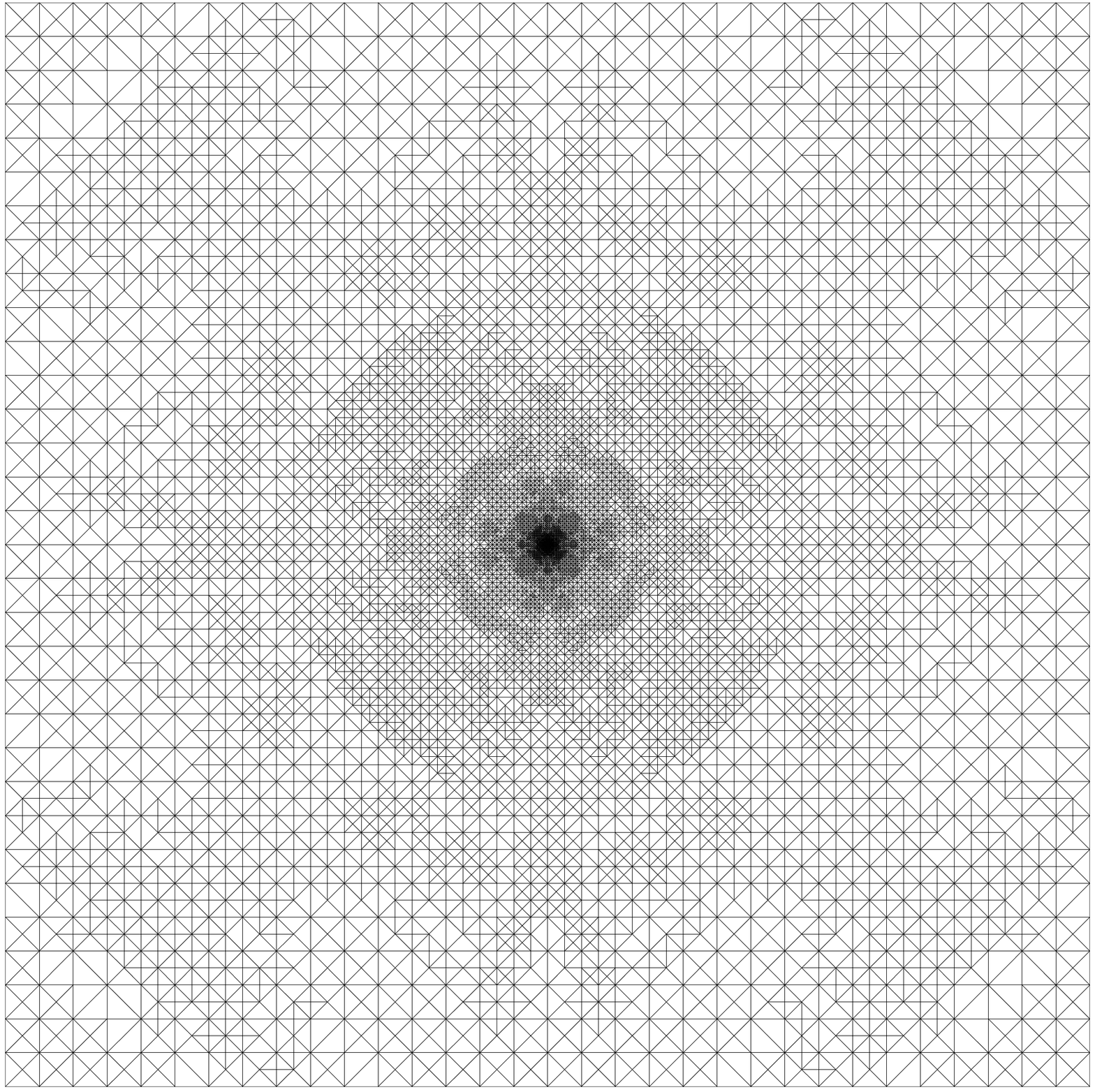}
\end{center}
\caption{\label{F:ex1_meshes}%
\small
\emph{Mesh grading and error norm for the fundamental solution.}  The triangulations after 20 iterations of \eqref{loop} for $\theta = 0.000$, $0.250$, $0.500$ (from left to right) illustrate that the mesh grading decreases with increasing $\theta$, which corresponds to weakening the error norm.}
\end{figure}
If $\theta$ is small, the mesh grading is very strong: for $\theta=0$, the triangles at the origin of meshes with about 1000 degrees of freedom (DOFs) have areas smaller than $10^{-16}$.  In the case of $\theta = 0.125$, this happens for meshes with about 5000 DOFs.

A next step in our numerical testing of the estimator $\eta_\theta$ could be to study the decay rate of the estimated error $|U-u|_{1-\theta,\Omega}$.  This will be done for the second, more involved example.  Here we shall instead study decay rates of two error notions, for which the estimator $\eta_\theta$ is not originally designed.  The first error notion is $|u-U|_{1,\Omega^0}$ with $\Omega^0 = \Omega\setminus B(0;\tfrac{1}{4}) = \{ (x,y) \in \Omega : |x|+|y| > \tfrac{1}{4}\}$.  Since $u\in H^2(\Omega^0)$, the maximum decay rate for it with linear finite elements is $\#\text{DOFs}^{-1/2}$, reached for example by uniform refinement.
The second error notion is the $L^2$-error $|u-U|_{0,\Omega}$.  Here we have $D^2u\in L^p(\Omega)$ for any $p\in(0,1)$ and thus $u$ is an element of the Besov space $B^2_p(L^p(\Omega))$.  Consequently, the maximum decay rate with linear finite elements is $\#\text{DOFs}^{-1}$, reached for example by thresholding \cite[Theorem~5.1]{Binev.Dahmen.DeVore.Petrushev:02}.  Figure~\ref{F:ex1_curves} suggests that, except for $\theta=0$, the meshes generated with $\eta_\theta$ provide the maximum decay rate for both errors.
\begin{figure}[h!tbp]
	\includegraphics[width=.47\textwidth]{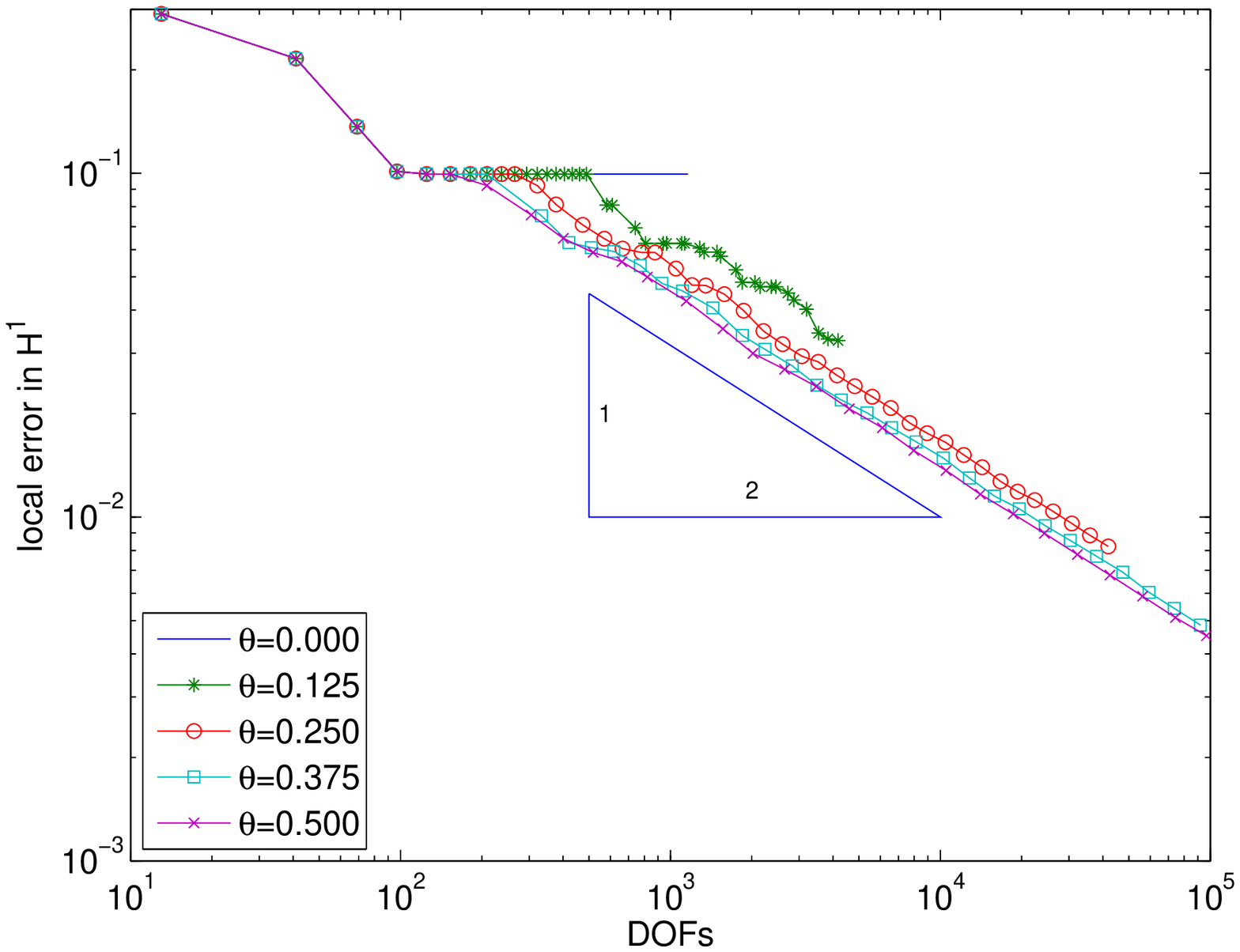}
	\includegraphics[width=.47\textwidth]{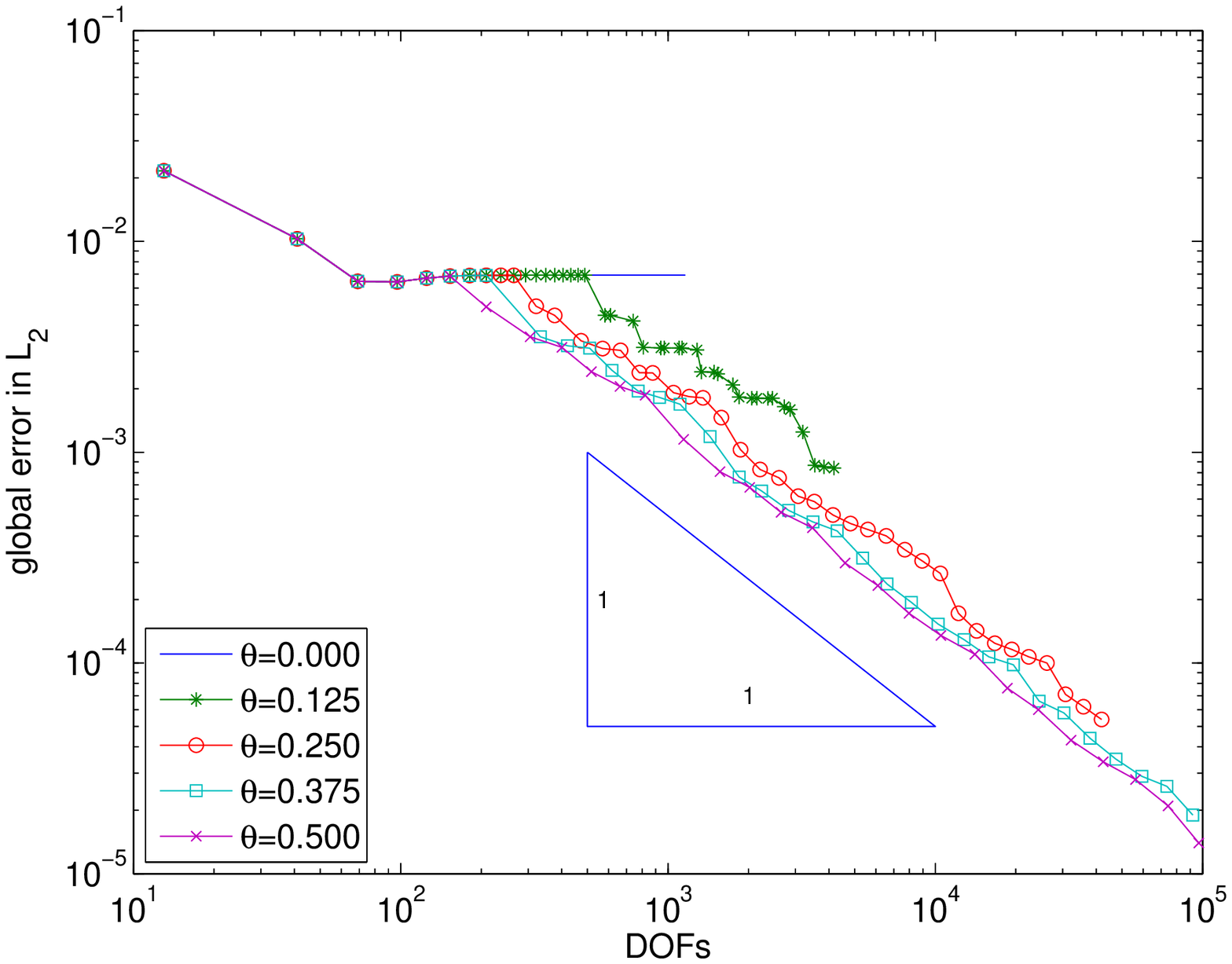}
	\caption{\label{F:ex1_curves}%
		\small
		\emph{Convergence histories of the $H^1$-error off the singularity (left) and the $L^2(\Omega)$-error (right) for the fundamental solution.}  In both cases, we plot error versus $\#\text{DOFs}$ in logarithmic scales.  The plots for $\theta=0$, $0.125$ and $0.25$ prematurely end because we encounter triangles whose areas are below $10^{-16}$.}
\end{figure}
We however notice an advantage for greater values of $\theta$, where the initial stagnation of the error decay is shorter.  This stagnation expresses the difference between the observed error notion and the estimator, which puts more importance to the singularity at the origin.
For $\theta = 0$, the fact that $u\not\in H^1(\Omega)$ appears to be reflected in an infinite stagnation.

\begin{example}[Point sources and reentrant corner]
\label{E:reentrant corner}
Consider the non-convex L-shaped domain $\Omega = (-1,1)^2 \setminus \big( [0,1)\times(-1,0] \big)$ and the boundary value problem
\begin{equation*}
\begin{aligned}
 -\Delta u
 &=
 \delta_{(0.33, 0.66)}
  + \delta_{(-0.251, -0.85)} + \delta_{(-0.25, -0.87)}
 \quad & &\text{in }\Omega,
\\
 u &= 0 & &\text{on }\partial\Omega.
\end{aligned}
\end{equation*}
\end{example}
The exact solution is not known to us;  see Figure \ref{F:ex2_solutions} for an approximate solution and corresponding triangulation. 
\begin{figure}[h!tbp]
	\begin{center}
		\includegraphics[width=.55\textwidth]{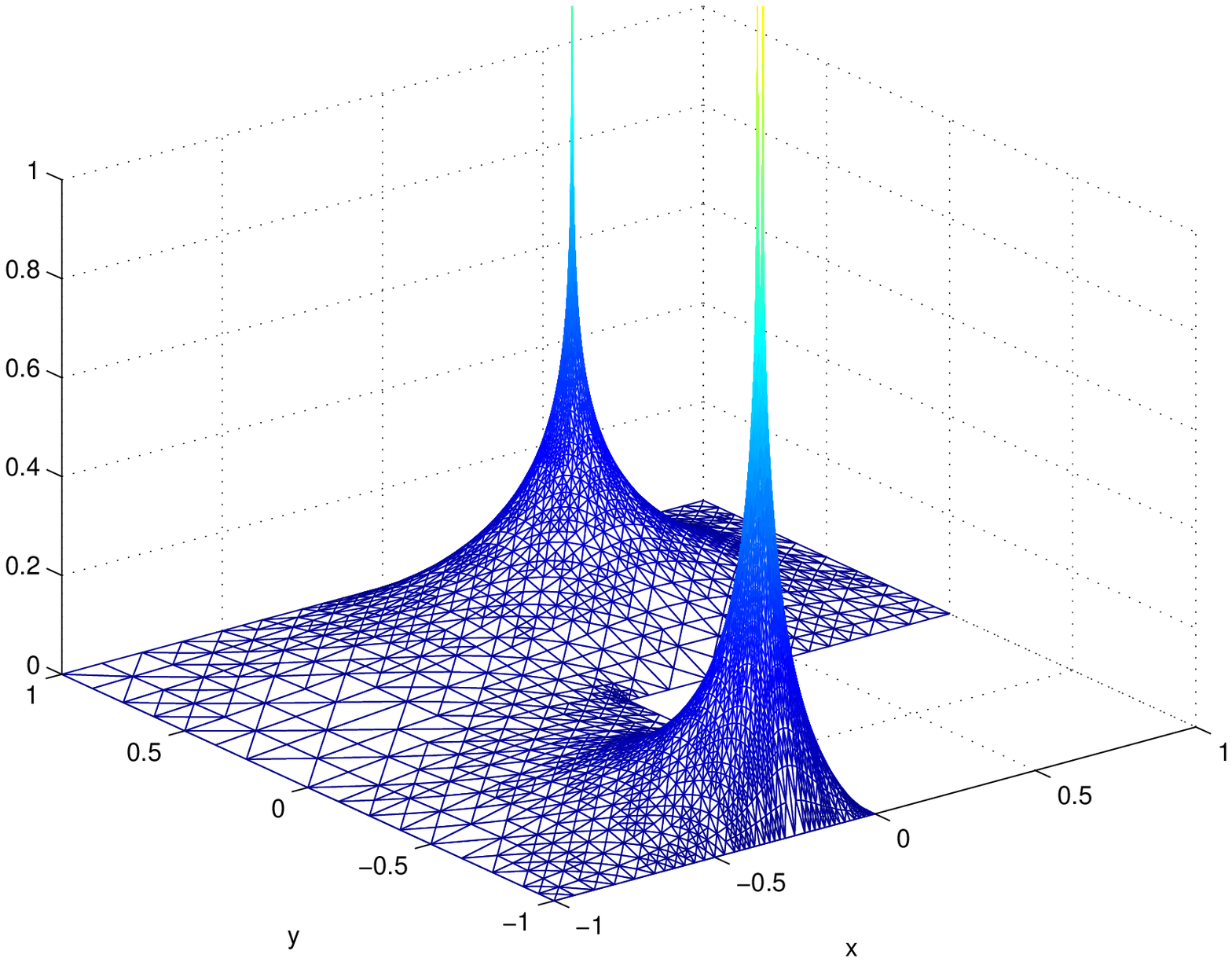}
		\hfil
		\includegraphics[width=.43\textwidth]{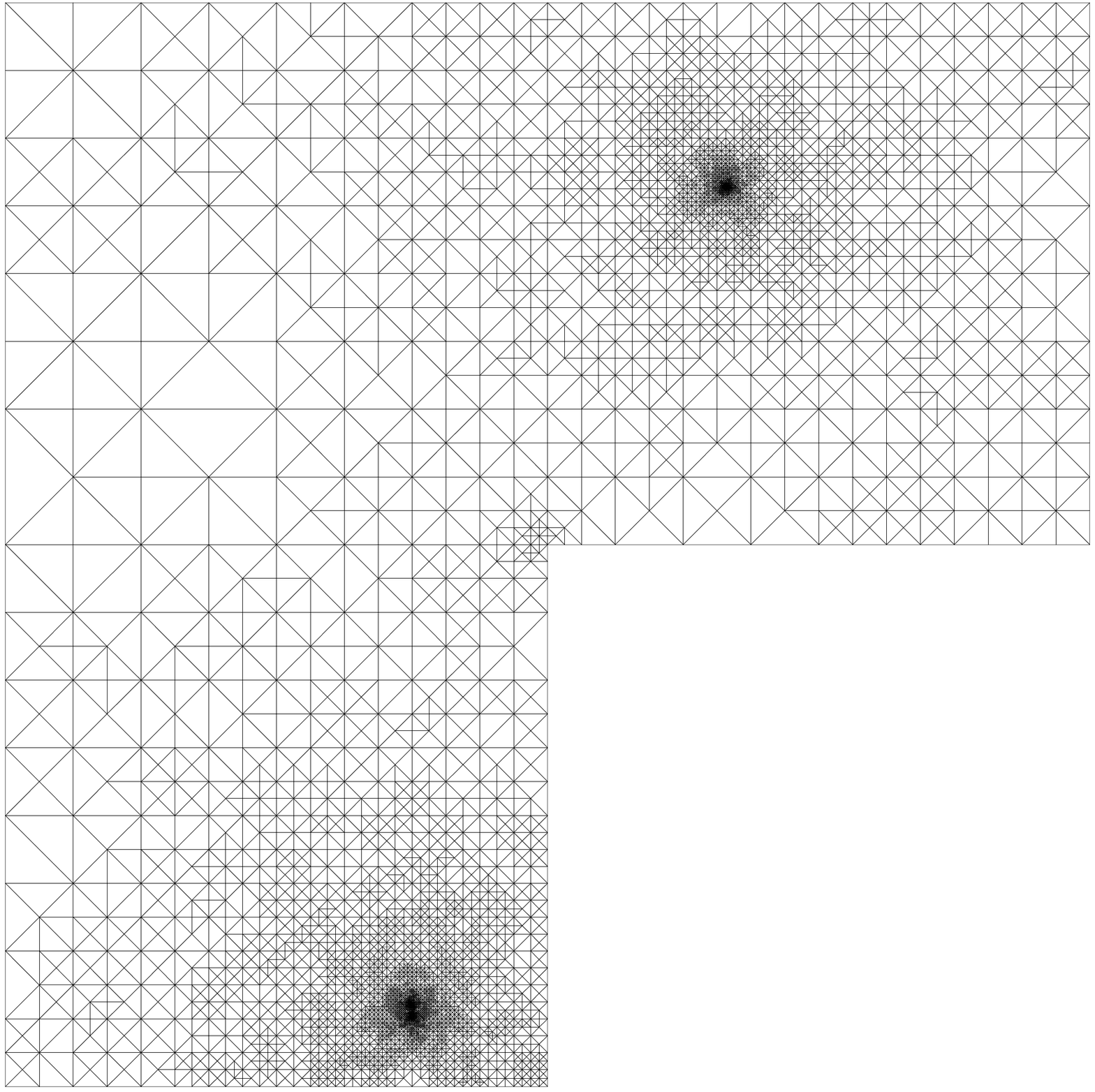}
		
	\end{center}
	\caption{\label{F:ex2_solutions}%
		\small
		\emph{Approximate solution and triangulation for Example~\ref{E:reentrant corner}}
		after 25 iterations with $\theta=0.375$.}
\end{figure}
Notice that boundary condition and right-hand side are compatible at the reentrant corner.  Nevertheless, the refinement at the reentrant corner indicates some non-compatibility of the error. The three point sources entail that the exact solution has three logarithmic-type singularities, two of them being very close.  As for Example \ref{E:fund-sol} and independently of data compatibility at the reentrant corner, we thus have $D^2u \in L^p(\Omega)$ for all $p\in(0,1)$.  For $\theta\in(0,1]$, the maximum decay rate for the approximation with linear finite elements in $|\cdot|_{1-\theta,\Omega}$ is therefore $\#\text{DOFs}^{-(1+\theta)/2}$, again reached for example by thresholding \cite[Theorem~5.1]{Binev.Dahmen.DeVore.Petrushev:02}.
Figure \ref{F:ex2_curves} indicates that these maximum decay rates are also obtained with the triangulation generated with the help of the estimator $\eta_\theta$; machine precision prevents a possibly better confirmation in the cases $\theta = 0.125$ and $0.25$.

\begin{figure}[h!tbp]
\begin{minipage}{.49\textwidth}
\includegraphics[width=.9\textwidth]{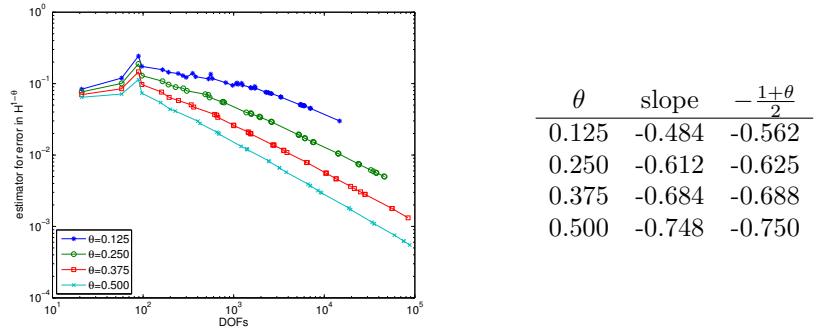}
\end{minipage}
\hfil
\begin{minipage}{.33\textwidth}
\begin{tabular}{ccc}
$\theta$ & slope & $-\frac{1+\theta}2$ \\
\hline
 0.125 &  -0.484  & -0.562 \\
 0.250 &  -0.612  & -0.625 \\
 0.375 &  -0.684  & -0.688 \\
 0.500 &  -0.748  & -0.750 
\end{tabular}
\end{minipage}
\caption{\label{F:ex2_curves}%
\small
\emph{Convergence histories of $H^{1-\theta}$-error} for Example \ref{E:reentrant corner}.  We plot the a posteriori estimator, which is equivalent to the $H^{1-\theta}$-error by \S\ref{S:apost-analysis}, versus $\#\text{DOFs)}$  in logarithmic scales.  The plots for $\theta=0$, $0.125$, and $0.25$ prematurely end because we encounter triangles whose areas are below $10^{-16}$. The table on the right shows the slope of the final part of each curve obtained by a least squares fit.}
\end{figure}

\end{document}